\documentclass{amsart}

\usepackage{amsmath}
\usepackage{amssymb}
\usepackage{mdframed}
\usepackage{amsthm}
\usepackage{mathtools}
\usepackage{enumitem}
\usepackage{placeins}
\usepackage{url}
\usepackage{hyperref}
\usepackage{physics}
\usepackage{amsmath}
\usepackage{tikz}
\usepackage{mathdots}
\usepackage{yhmath}
\usepackage{cancel}
\usepackage{color}
\usepackage{siunitx}
\usepackage{array}
\usepackage{multirow}
\usepackage{amssymb}
\usepackage{gensymb}
\usepackage{tabularx}
\usepackage{extarrows}
\usepackage{booktabs}
\usetikzlibrary{fadings}
\usetikzlibrary{patterns}
\usetikzlibrary{shadows.blur}
\usetikzlibrary{shapes}

\newtheorem{thm}{Theorem}
\newtheorem{obs}[thm]{Observation}
\newtheorem{lem}[thm]{Lemma}
\newtheorem{prop}[thm]{Proposition}
\newtheorem{cor}[thm]{Corollary}

\newtheorem*{T1}{Theorem~\ref{thm:maintheorem}}
\newtheorem*{T2}{Theorem~\ref{thm:modifiedProjectionTheorem}}

\newtheorem*{T5}{Theorem~\ref{thm:packingThm}}
\newtheorem*{T6}{Theorem~\ref{thm:regularYFullDim}}

\newtheorem*{T8}{Theorem~\ref{thm:moregeneralmaintheorem}}

\theoremstyle{remark}
\newtheorem*{remark}{Remark}

\DeclareMathOperator{\Dim}{Dim}
\DeclareMathOperator{\dimH}{dim_H}
\DeclareMathOperator{\dimP}{dim_P}

\newcommand{\R}{\mathbb{R}}

\newcommand{\N}{\mathbb{N}}
\newcommand{\Q}{\mathbb{Q}}

\newcommand{\ve}{\varepsilon}
\newcommand{\uhr}{{\upharpoonright}}

\newenvironment{proofof}[1]{\begin{trivlist}
		\item[\hskip \labelsep \textit{Proof of #1.}]}{\end{trivlist}}

\begin{document}
\title{Dimension of Pinned Distance Sets for Semi-Regular Sets}
\author{Jacob B. Fiedler}
\address{Department of Mathematics, University of Wisconsin, Madison, Wisconsin 53715}

\email{jbfiedler2@wisc.edu}
\thanks{The first author was supported in part by NSF DMS-2037851 and NSF DMS-2246906. Both authors are grateful to the American Institute of Mathematics for hosting the workshop \textit{Effective methods in measure and dimension}, which was the genesis of this collaboration.}

\author{D. M. Stull}
\address{Department of Mathematics, University of Chicago, Chicago, IL 60637}
\email{dmstull@uchicago.edu}

	\maketitle

 \begin{abstract}
    We prove that if $E\subseteq \R^2$ is analytic and $1<d < \dim_H(E)$, there are ``many'' points $x\in E$ such that the Hausdorff dimension of the pinned distance set $\Delta_x E$ is at least $d\left(1 - \frac{\left(D-1\right)\left(D-d\right)}{2D^2+\left(2-4d\right)D+d^2+d-2}\right)$, where $D = \dim_P(E)$. In particular, we prove that $\dim_H(\Delta_x E) \geq \frac{d(d-4)}{d-5}$ for these $x$, which gives the best known lower bound for this problem when $d \in (1, 5-\sqrt{15})$. We also prove that there exists some $x\in E$ such that the packing dimension of $\Delta_x E$ is at least $\frac{12 -\sqrt{2}}{8\sqrt{2}}$. Moreover, whenever the packing dimension of $E$ is sufficiently close to the Hausdorff dimension of $E$, we show the pinned distance set $\Delta_x E$ has full Hausdorff dimension for many points $x\in E$; in particular the condition is that $D<\frac{(3+\sqrt{5})d-1-\sqrt{5}}{2}$.
    
    We also consider the pinned distance problem between two sets $X, Y\subseteq \R^2$, both of Hausdorff dimension greater than 1. We show that if either $X$ or $Y$ has equal Hausdorff and packing dimensions, the pinned distance $\Delta_x Y$ has full Hausdorff dimension for many points $x\in X$. 
 \end{abstract}

\section{Introduction}

Given some $E\subseteq \mathbb{R}^n$ of a particular size, a natural question is whether one can bound the size of its distance set, defined as 
\begin{equation*}
\Delta E =\{\vert x - y\vert: x, y\in E\}
\end{equation*}
The continuous version of this question is the \textit{Falconer distance problem}, concerning the fractal dimension of $E$. Falconer conjectured that for Borel sets $E$, whenever $\dim_H(E)>\frac{n}{2}$, then $\Delta E$ has positive measure. In the plane, the best known bound is that $\dim_H(E)>\frac{5}{4}$ implies $\Delta E$ has positive measure and is due to Guth, Iosevich, Ou, and Wang \cite{GutIosOuWang20}. In higher dimensions, very recent work of Du, Ou, Ren, and Zhang established the threshold that $\dim_H(E)>\frac{n}{2}+\frac{1}{4}-\frac{1}{8n+4}$ implies $\Delta E$ has positive measure \cite{DuOuRenZhang23a}\cite{DuOuRenZhang23b}\cite{Ren23}.

In fact, these works each established a stronger result concerning \emph{pinned} distance sets. For $x\in \mathbb{R}^n$, we define
\begin{equation*}
\Delta_x E =\{\vert x - y\vert: y\in E\}.
\end{equation*}
The aforementioned papers prove that under the respective assumptions, there exists some $x\in E$ such that $\Delta_x E$ has positive measure. Note that such a result for pinned distance sets immediately implies the corresponding result for distance sets.

A closely related problem is to prove lower bounds on the Hausdorff or packing dimension of $\Delta_x E$ for ``many'' $x\in E$, given $d=\dim_H(E)$. Restricting consideration to $\mathbb{R}^2$ for the remainder of the paper, we now discuss a few previous bounds of this type. For $d\in(1, \frac{5}{4})$, Liu proved that $\dim_H(\Delta_x E)>\frac{4d}{3}-\frac{2}{3}$ in \cite{Liu20}. Shmerkin proved a better bound for $d$ not much greater than 1, namely that $\dim_H(\Delta_x E)\geq \frac{2}{3}+\frac{1}{42}$ \cite{Shmerkin20}. The second author improved the best known bound for $d$ not much larger than 1, proving that $\dim_H(\Delta_x E)\geq\frac{d}{4}+\frac{1}{2}$ \cite{Stull22c}. 

Complementing the bounds cited in the previous paragraph, Du, Ou, Ren and Zhang proved a bound for $d$ \emph{less} than the dimension threshold of Falconer's conjecture \cite{DuOuRenZhang23a}. In $\mathbb{R}^2$, they showed $\sup_{x\in E}\dim_H(\Delta_x E)\geq\frac{5d}{3}-1$. As for packing dimension, in \cite{KelShm19}, Shmerkin and Keleti proved that
\begin{equation*}
\dim_P(\Delta_x E)\geq \frac{1}{4}\left(1 + d + \sqrt{3d(2-d)}\right).
\end{equation*}
Finally, we note that Shmerkin proved that for sets $E$ which are \emph{regular} in the sense that $\dim_H(E)=\dim_P(E)$, so long as $\dim_H(E)>1$, then for most $x$, the pinned distance set has full dimensions, i.e., $\dim_H(\Delta_x E)=1$ \cite{Shmerkin19}.\footnote{Observe that this regularity is weaker than Alfors-David regularity, which Orponen considered in \cite{Orponen17}. Throughout the remainder of the paper, by regular, we mean the more general notion.}

Our work makes a number of improvements to the pinned distance problem in the plane. First, we are able to prove a dimensional lower bound which takes into account the packing dimension of $E$.

\begin{thm}\label{thm:maintheorem}
Let $E\subseteq \R^2$ be analytic such that $1<d <\dim_H(E)$. Then there is a subset $F \subseteq E$ of full dimension such that 
\begin{equation*}
    \dim_H(\Delta_x E) \geq d\left(1 - \frac{\left(D-1\right)\left(D-d\right)}{2D^2+\left(2-4d\right)D+d^2+d-2}\right),
\end{equation*}
for all $x\in F$, where $D = \dim_P(E)$. In particular, $\dim_H(E\setminus F)\leq d<\dim_H(E)$. Furthermore, if 
\begin{equation*}
D<\frac{(3+\sqrt{5})d-1-\sqrt{5}}{2}
\end{equation*}
Then $\dim_H(\Delta_x E)=1$.
\end{thm}

Note that the second part of the theorem significantly generalizes Shmerkin's result for regular sets $E\subseteq\mathbb{R}^2$, in the sense that we show that $E$ need not be fully regular to have most of its pinned distance sets be full dimension. Instead, it only needs to have sufficiently close Hausdorff and packing dimension, a form of ``semi-regularity''. As a corollary of this theorem, we obtain an improvement over the second author's previous bound, namely. 

\begin{cor}\label{cor:firstMainCor}
    Let $E\subseteq \R^2$ be analytic such that $1 < d<\dim_H(E) $. Then there is a subset $F \subseteq E$ of full dimension such that 
\begin{equation*}
    \dim_H(\Delta_x E) \geq \frac{d(d-4)}{d-5}.
\end{equation*}
for all $x\in F$.
\end{cor}

Additionally, we can bound the dimension of the pinned distance sets in terms of only the packing dimension, as below.

\begin{cor}\label{cor:secondMainCor}
Let $E\subseteq \R^2$ be analytic such that $\dim_H(E) > 1$. Then, for all $x\in E$ outside a set of (Hausdorff) dimension one,
\begin{equation*}
    \dim_H(\Delta_x E) \geq \frac{D+1}{2D},
\end{equation*}
where $D = \dim_P(E)$.
\end{cor}

This corollary turns out to be useful in establishing the following improvement on Shmerkin and Keleti's packing dimension bound.\footnote{More precisely, we use its effective analog.}

\begin{thm}\label{thm:packingThm}
Let $E\subseteq \R^2$ be analytic such that $\dim_H(E) > 1$. Then there exists some $x\in E$ such that,
\begin{equation*}
    \dim_P(\Delta_x E) \geq \frac{12 -\sqrt{2}}{8\sqrt(2)}\approx 0.9356.
\end{equation*}
\end{thm}

However, our results are more general than the above. We are able to prove essentially the same bounds in the case that our pinned points $x$ lie in some set $X$ and we consider the set of distances from $x$ to some analytic set $Y$.  Theorem \ref{thm:maintheorem} is itself an immediate corollary of the following more general theorem.

\begin{thm}\label{thm:moregeneralmaintheorem}
Let $Y\subseteq \R^2$ be analytic such that $1<d_y =\dim_H(Y)$ and $D_y=\Dim_p(Y)$. Let $X\subseteq \R^2$ be such that $1<d_x <\dim_H(X) $ and $D_x=\Dim_p(X)$. Then there is some $F \subseteq X$ of full dimension such that 
\begin{equation*}
    \dim_H(\Delta_x E) \geq d\left(1 - \frac{\left(D-1\right)\left(D-d\right)}{2D^2+\left(2-4d\right)D+d^2+d-2}\right),
\end{equation*}
for all $x\in F$, where $d=\min\{d_x, d_y\}$ and $D=\max\{D_x, D_y\}$. In particular, $\dim_H(X\setminus F)\leq d_x<\dim_H(X)$  Furthermore, if 
\begin{equation*}
D<\frac{(3+\sqrt{5})d-1-\sqrt{5}}{2}
\end{equation*}
Then $\dim_H(\Delta_x E)=1$.
\end{thm}

The second portion of this theorem amounts to a semi-regularity condition on both $X$ and $Y$. Our work also shows that we can achieve full dimension pinned distance sets when all that we require of $X$ is $\dim_H(X)>1$, at the cost of a somewhat more strict semi-regularity condition on $Y$.

\begin{thm}\label{thm:regularYFullDim}
Let $Y\subseteq\R^2$ be analytic with $\dim_H(Y) > 1$ and $\dim_P(Y) < 2 \dim_H(Y)-1$. Let $X \subseteq \R^2$ be any set such that $\dim_H(X) > 1$. Then for all $x\in X$ outside a set of (Hausdorff) dimension one,
\begin{equation*}
    \dim_H(\Delta_x Y) = 1.
\end{equation*}
\end{thm}

Finally, we are able to show that regularity in just the pin set $X$ is good enough to imply that typical pinned distance sets have dimension 1.

\begin{thm}\label{thm:regularXFullDim}
Let $Y\subseteq\R^2$ be analytic with $\dim_H(Y) > 1$. Let $X \subseteq \R^2$ be any set such that $\dim_H(X) = \dim_P(X) > 1$. Then there is a subset $F \subseteq X$ such that, 
\begin{equation*}
    \dim_H(\Delta_x Y) = 1,
\end{equation*}
for all $x \in F$. Moreover, $\dim_H(X\setminus F) < \dim_H(X)$.
\end{thm}

Now, we outline the structure of the paper. This paper employs recently developed ``effective'' methods in the study of dimension. These methods connect Hausdorff and packing dimension to Kolmogorov complexity through \textit{point-to-set principles}, and thus open classical problems up to attack by tools from computability theory. Section 2 covers the necessary preliminaries from computability theory, as well as a few less introductory but crucial lemmas. 

Section 3 deals with the proof of an effective projection theorem for points $x$ which is primarily used to bound the growth of the complexity of $\vert x - y \vert$ in the next section. In order to obtain this projection theorem, we need to perform a certain partitioning argument on the interval $[1, r]$ considering the complexity function $K_r(x)$. Partitioning $[1, r]$ into smaller intervals so that the complexity function has useful properties on each interval, or even just certain intervals, is a recurring idea throughout this paper. 

Section 4 is the main thrust of the argument on the effective side. The idea is to now partition $[1, r]$ so that $K_r(\vert x - y\vert )$ either grows at an optimal rate of 1 or grows at an average rate at least equal to the average growth rate of $K_r(y)$ on each interval. First, in section 4.2, we construct a partition which only uses the first kind of interval; this does not require the application of the projection theorem and will thus be essential in proving Theorem \ref{thm:regularYFullDim}.\footnote{Recall that Theorem \ref{thm:regularYFullDim} was more or less independent of the dimension of $X$, so long as it is greater than 1. This is why we do not use the projection theorem here, as it concerns the pin $x$ and its effective dimension.} Section 4.3 details the construction of a more general partition that achieves better bounds using the projection theorem, and section 4.4 sums over this partition to obtain the effective analog of Theorem \ref{thm:moregeneralmaintheorem}. This effective analog is a bound on the complexity of the distance at \emph{every} precision, and in section 4.5 we use it as a basis to obtain improved bounds at certain well-chosen precisions; this yields the effective analog of Theorem \ref{thm:packingThm}. 

Section 5 is where we perform the reductions to the classical results. Essentially, we have to show that sets of the given dimensions always have points $x$ and $y$ with the desired algorithmic properties, which then imply the bounds of our effective theorems hold for certain distances. Performing these reductions yields Theorem \ref{thm:packingThm}, Theorem \ref{thm:moregeneralmaintheorem}, and Theorem \ref{thm:regularYFullDim}. 

Section 6, wherein the goal is to prove Theorem \ref{thm:regularXFullDim}, is more self-contained. The idea is that if the point $x$ is regular in the sense of having equal effective Hausdorff and effective packing dimensions, we get an essentially optimal effective projection theorem. This allows us to take intervals as long as we want when partitioning $\vert x-y\vert$, which makes it straightforward to establish the bound of 1. A complication when performing the reduction to the classical result is that regular sets do not necessarily contain sufficiently many regular points, so we need a variant of the projection theorem that holds for $x$'s that are \emph{almost} regular. As a consequence, we cannot take arbitrarily long intervals when partitioning $\vert x-y\vert$, but as $\dim_H(Y)>d_y>1$, we do not need arbitrarily long intervals to get the bound of 1. Thus, the reduction goes through.
\begin{remark}
For readers who want a relatively straightforward demonstration of the main ideas of this paper, we suggest considering starting with section 6. The case of almost regular $x$, while it does not follow from the previous sections, has a similar structure to sections 3-5 without as many of the complications.
\end{remark}

\bigskip

\section{Preliminaries}\label{sec:prelim}

\subsection{Kolmogorov complexity and effective dimension}
The \emph{conditional Kolmogorov complexity} of binary string $\sigma\in\{0,1\}^*$ given a binary string $\tau\in\{0,1\}^*$ is the length of the shortest program $\pi$ that will output $\sigma$ given $\tau$ as input. Formally, the conditional Kolmogorov complexity of $\sigma$ given $\tau$ is
\[K(\sigma\mid\tau)=\min_{\pi\in\{0,1\}^*}\left\{\ell(\pi):U(\pi,\tau)=\sigma\right\}\,,\]
where $U$ is a fixed universal prefix-free Turing machine and $\ell(\pi)$ is the length of $\pi$. Any $\pi$ that achieves this minimum is said to \emph{testify} to, or be a \emph{witness} to, the value $K(\sigma\mid\tau)$. The \emph{Kolmogorov complexity} of a binary string $\sigma$ is $K(\sigma)=K(\sigma\mid\lambda)$, where $\lambda$ is the empty string.	We can easily extend these definitions to other finite data objects, e.g., vectors in $\Q^n$, via standard binary encodings. See~\cite{LiVit08} for details.

The \emph{Kolmogorov complexity} of a point $x\in\R^m$ at \emph{precision} $r\in\N$ is the length of the shortest program $\pi$ that outputs a \emph{precision-$r$} rational estimate for $x$. Formally, this is 
\[K_r(x)=\min\left\{K(p)\,:\,p\in B_{2^{-r}}(x)\cap\Q^m\right\}\,,\]
where $B_{\ve}(x)$ denotes the open ball of radius $\ve$ centered on $x$. Note that this implies that the Kolmogorov complexity of a point is non-decreasing in precision. The \emph{conditional Kolmogorov complexity} of $x$ at precision $r$ given $y\in\R^n$ at precision $s\in\R^n$ is
\[K_{r,s}(x\mid y)=\max\big\{\min\{K_r(p\mid q)\,:\,p\in B_{2^{-r}}(x)\cap\Q^m\}\,:\,q\in B_{2^{-s}}(y)\cap\Q^n\big\}\,.\]
When the precisions $r$ and $s$ are equal, we abbreviate $K_{r,r}(x\mid y)$ by $K_r(x\mid y)$. As a matter of notational convenience, if we are given a non-integral positive real as a precision parameter, we will always round up to the next integer. Thus $K_{r}(x)$ denotes $K_{\lceil r\rceil}(x)$ whenever $r\in(0,\infty)$.

A basic property, proven by Case and J. Lutz~\cite{CasLut15} shows that the growth rate of the Kolmogorov complexity of a point is essentially bounded by the dimension of the ambient space. Since this paper concerns $\mathbb{R}^2$, we will frequently use this in the form that for any $\ve>0$, for sufficiently large $s$ we have that
\begin{equation*}
K_{r+s}(x)\leq K_r(x)+2 s + \ve s
\end{equation*}

We may \emph{relativize} the definitions in this section to an arbitrary oracle set $A \subseteq \N$. We will frequently consider the complexity of a point $x \in \R^n$ \emph{relative to a point} $y \in \R^m$, i.e., relative to an oracle set $A_y$ that encodes the binary expansion of $y$ is a standard way. We then write $K^y_r(x)$ for $K^{A_y}_r(x)$. Oracle access to the \emph{entire} binary expansion of a point is no less useful than conditional access to that binary expansion only up to a certain precision. Thus, we note that, for every $x\in\R^n$ and $y\in\R^m$,
\begin{equation}\label{eq:OraclesDontIncrease}
K_{s,r}(x\mid y)\geq K^y_s(x) - O(\log r) - O(\log s),
\end{equation}
for every $s, r\in\N$

One of the most useful properties of Kolmogorov complexity is that it obeys the \emph{symmetry of information}. That is, for every $\sigma, \tau \in\{0,1\}^*$,
\[K(\sigma, \tau) = K(\sigma) + K(\tau \mid \sigma, K(\sigma)) + O(1)\,.\]

We also have the more technical lemmas detailing versions of symmetry of information that hold for Kolmogorov complexity in $\R^n$.  Lemma~\ref{lem:unichain} was proved in the second author's work~\cite{LutStu20}.
\begin{lem}[\cite{LutStu20}]\label{lem:unichain}
	For every $m,n\in\N$, $x\in\R^m$, $y\in\R^n$, and $r,s\in\N$ with $r\geq s$,
	\begin{enumerate}
		\item[\textup{(i)}]$\displaystyle |K_r(x\mid y)+K_r(y)-K_r(x,y)\big|\leq O(\log r)+O(\log\log \vert y\vert)\,.$
		\item[\textup{(ii)}]$\displaystyle |K_{r,s}(x\mid x)+K_s(x)-K_r(x)|\leq O(\log r)+O(\log\log\vert x\vert)\,.$
	\end{enumerate}
\end{lem}
	
A consequence of Lemma \ref{lem:unichain}, is the following. 
\begin{lem}[\cite{LutStu20}]\label{lem:symmetry}
	Let $m,n\in\N$, $x\in\R^m$, $z\in\R^n$, $\ve > 0$ and $r\in\N$. If $K^x_r(z) \geq K_r(z) - O(\log r)$, then the following hold for all $s \leq r$.
	\begin{enumerate}
		\item[\textup{(i)}]$\displaystyle K^x_s(z) \geq K_s(z) - O(\log r)\,.$
		\item[\textup{(ii)}]$\displaystyle K_{s, r}(x \mid z) \geq K_s(x)- O(\log r)\,.$
	\end{enumerate}
\end{lem}

\bigskip

J. Lutz~\cite{Lutz03a} initiated the study of effective dimensions (also known as \emph{algorithmic dimensions}) by effectivizing Hausdorff dimension using betting strategies called~\emph{gales}, which generalize martingales.  Mayordomo showed that effective Hausdorff dimension can be characterized using Kolmogorov complexity~\cite{Mayordomo02}. In this paper, we use this characterization as a definition. The \emph{effective Hausdorff dimension}  of a point $x\in\R^n$ is
\[\dim(x)=\liminf_{r\to\infty}\frac{K_r(x)}{r}\,.\]
The \emph{effective packing dimension} of a point $x\in\R^n$ is
\[\Dim(x)=\limsup_{r\to\infty}\frac{K_r(x)}{r}\,.\]
We can relativize both definitions, so that the effective Hausdorff and packing dimension \textit{with respect to an oracle} $A\subseteq \N$ are
\begin{center}
    $\dim^A(x)=\liminf_{r\to\infty}\frac{K^A_r(x)}{r}$ and $\Dim^A(x)=\limsup_{r\to\infty}\frac{K^A_r(x)}{r}$
\end{center}

\bigskip

\subsection{The Point-to-Set Principle}\label{subsec:ptsp}
The \emph{point-to-set principle} shows that the Hausdorff and packing dimension of a set can be characterized by the effective Hausdorff and effecive packing dimension of its individual points. Specifically, J. Lutz and N. Lutz~\cite{LutLut18} showed the following for arbitrary subsets of $\mathbb{R}^n$.
\begin{thm}[Point-to-set principle~\cite{LutLut18}]\label{thm:p2s}
Let $n \in \N$ and $E \subseteq \R^n$. Then
\begin{equation*}
\dimH(E) = \adjustlimits\min_{A \subseteq \N} \sup_{x \in E} \dim^A(x).
\end{equation*}
\begin{equation*}
\dimP(E) = \adjustlimits\min_{A \subseteq \N} \sup_{x \in E} \Dim^A(x).
\end{equation*}
\end{thm}

Stated as above, it is clear that Hausdorff and pacing dimension are in a certain respect dual to each other. The only difference is a limit inferior versus a limit superior for the individual points. This immediately implies that the packing dimension of a set is no less than its Hausdorff application. 

The general point-to-set principle is extremely useful, but for some applications, we would like to either remove the oracle, or at least be able to say something about which oracles achieve the minimum. The first point-to-set principle for Hausdorff dimension, which holds for a restricted class of sets, was implicitly proven by Hitchcock~\cite{Hitchcock05} and J. Lutz~\cite{Lutz03a}. 

A set $E \subseteq \R^n$ is \textit{effectively compact relative to} $A$ if the set of finite open covers of $E$ by rational balls is computably enumerable relative to $A$.\footnote{The balls are rational in the sense that the coordinates of the centers and the radii are all rational numbers, which allows us to identify each ball by a finite string.} We will use the fact that every compact set is effectively compact relative to some oracle. 
\begin{thm}[\cite{Hitchcock05, Lutz03a}]\label{thm:strongPointToSetDim}
Let $E \subseteq \R^n$ and $A \subseteq \N$ be such that $E$ is effectively compact relative to $A$. Then
\[\dimH(E) = \sup\limits_{x \in E} \dim^A(x)\,.\]
\end{thm}

\begin{remark}
    We only state this restricted point-to-set principle for Hausdorff dimension because it is known that it fails for packing dimension, see \cite{Conidis08}. Informally, this can be seen to occur because effective compactness and Hausdorff dimension both deal with covers, whereas it is hard to convert the covers we have from effective compactness into usable information about packings.
\end{remark}
In order to apply Theorem \ref{thm:strongPointToSetDim} to the pinned distance sets, we need the following fact of computable analysis.

\begin{obs}\label{obs:effcompact}
Let $E \subseteq \R^2$ be a compact set and let $A\subseteq\N$ be an oracle  relative to which $E$ is effectively compact. Then, for every $x\in\R^2$, $\Delta_x E$ is effectively compact relative to $(x, A)$.
\end{obs}

\bigskip

\subsection{Helpful lemmas}\label{subsec:primarylemmas}
In this section, we recall several lemmas which were introduced by Lutz and Stull \cite{LutStu20, LutStu18} and which will be used throughout the paper. Note that these lemmas each relativize with the addition of an oracle $A$. The first lemma shows that the precision to which we can compute $e$ given $x, w$ such that $p_e x = p_e w$ depends linearly on the distance of $x$ and $w$.
\begin{lem}[\cite{LutStu20}]\label{lem:lowerBoundOtherPoints}
Let $z \in \R^2$, $e \in S^{1}$, and $r \in \N$. Let $w \in \R^2$ such that  $p_e z = p_e w$ up to precision $r$.\footnote{This lemma was originally proven without the ``up to precision $r$'' qualifier, but we rephrase like this to match the form we will use the lemma in. The generalization is essentially immediate, because in this case there will be some sufficiently close point to $w$ with \emph{exactly} the same projection as $z$, indistinguishable from $w$ at precision $r$.} Then 
\begin{equation*}
    K_r(w) \geq K_t(z) + K_{r-t,r}(e\mid z) + O(\log r)\,,
\end{equation*}
where $t := -\log \vert z-w\vert$.
\end{lem}

We will commonly need to lower the complexity of points at specified positions. The following lemma shows that conditional complexity gives a convenient way to do this.
\begin{lem}[\cite{LutStu20}]\label{lem:oracles}
Let $z\in\R^2$, $\eta \in\Q_+$, and $r\in\N$. Then there is an oracle $D=D(r,z,\eta)$ with the following properties.
\begin{itemize}
\item[\textup{(i)}] For every $t\leq r$,
\[K^D_t(z)=\min\{\eta r,K_t(z)\}+O(\log r)\,.\]
\item[\textup{(ii)}] For every $m,t\in\N$ and $y\in\R^m$,
\[K^{D}_{t,r}(y\mid z)=K_{t,r}(y\mid z)+ O(\log r)\,,\]
and
\[K_t^{z,D}(y)=K_t^z(y)+ O(\log r)\,.\]
\item[\textup{(iii)}] If $B\subseteq\N$ satisfies $K^B_r(z) \geq K_r(z) - O(\log r)$, then \[K_r^{B,D}(z)\geq K_r^D(z) - O(\log r)\,.\]
\item[\textup{(iv)}] For every $t\in\N$, $u\in\R^n, w\in\R^m$
\[K_{r,t}(u\mid w) \leq K^D_{r,t}(u\mid w) + K_r(z) - \eta r + O(\log r)\,.\]
\end{itemize}
In particular, this oracle $D$ encodes $\sigma$, the lexicographically first time-minimizing witness to $K(z\uhr r\mid z\uhr s)$, where $s = \max\{t \leq r \, : \, K_{t-1}(z) \leq \eta r\}$.
\end{lem}

The final lemma in this section is a crucial tool at several points of the argument. Under certain conditions, it lets us lower bound the complexity growth of the $\vert x-y\vert$ by the complexity growth of $y$ on particular intervals.

\begin{lem}\label{lem:pointDistance}
Suppose that $x, y\in\R^2$, $t<r\in\N$, and $\eta, \ve\in\Q_+$ satisfy the following conditions.
\begin{itemize}
\item[\textup{(i)}]$K_r(y)\leq \left(\eta +\frac{\ve}{2}\right)r$.
\item[\textup{(ii)}] For every $w \in B_{2^{-t}}(y)$ such that $\vert x-y\vert = \vert x-w\vert$, \[K_{r}(w)\geq \eta r + \min\{\ve r, r-s -\ve r\}\,,\]
where $s=-\log\vert y-w\vert\leq r$.
\end{itemize}
Then for every oracle set $A\subseteq\N$,
\[K_{r,t}^{A, x}(y \mid y) \leq K^{A,x}_{r,t}( \vert x-y\vert\mid y) + 3\ve r + K(\ve,\eta)+O(\log r)\,.\]
\end{lem}

\bigskip

\section{Projection theorem}\label{sec:projections}

The main goal of this section is to prove the following projection theorem:

\begin{thm}\label{thm:modifiedProjectionTheorem}
Let  $x \in \R^2$, $e \in \mathcal{S}^1$, $\ve\in \Q^+$, $C\in\N$, $A\subseteq\mathbb{N}$, and $t, r \in \N$. Suppose that $r$ is sufficiently large, and that the following hold.
\begin{enumerate}
\item[\textup{(P1)}] $1 < d \leq \dim^A(x) \leq \Dim^A(x) \leq  D$.
\item[\textup{(P2)}] $ t \geq  \frac{d(2-D)}{2}r$.
\item[\textup{(P3)}] $K^{x, A}_s(e) \geq s - C\log s$, for all $s \leq t$. 
\end{enumerate}
Then 
\begin{equation*}
K^A_r(x \,|\, p_e x, e) \leq \max\{\frac{D-1}{D}(dr - t) + K^A_r(x) - dr, K^A_r(x) - r\} + \ve r.
\end{equation*}
\end{thm}
This projection has somewhat more restrictive hypotheses than the projection theorem of \cite{Stull22c}. Namely, depending on $D$, we may have a rather large lower bound on $t$. However, in the next section when we will need to apply this projection theorem, if $t$ is smaller than the above, it is easy to deduce the desired result without reference to this theorem. We will begin by introducing some definitions and tools from \cite{Stull22c} which will be of use in proving the modified projection theorem. 

\bigskip

\subsection{Projection preliminaries}
We need to consider $K^A_s(x)$ as a function of $s$. For convenience, we define $f:\R_+ \rightarrow \R_+$ to be the piece-wise linear function which agrees with $K^A_s(x)$ on the integers such that 
\begin{center}
$f(a) = f(\lfloor a\rfloor) + (a -\lfloor a\rfloor)(f(\lceil a \rceil) - f(\lfloor a\rfloor)) $,
\end{center}
for any non-integer $a$. Note that $f$ is non-decreasing since $K^A_r(x)$ is, and, for every $a < b\in\N$,
\begin{center}
$f(b) - f(a) \leq 2(b-a) + O(\log \lceil b \rceil)$. 
\end{center}
That is, the maximal growth rate of $f$ on large intervals is about 2.

There are several specials kinds of intervals on which we can bound the complexity of projections.

\begin{itemize}
\item An interval $[a,b]$ is called \textbf{\textit{teal}} if $f(b) - f(c) \leq b-c$ for every $a\leq c\leq b$.

\item An interval $[a, b]$ is called \textbf{\textit{yellow}} if $f(c) - f(a) \geq c - a$ for every $a\leq c \leq b$. 
\end{itemize}

More concretely, these intervals are useful due to the following proposition, which is Corollary 16 in \cite{Stull22c}:

\begin{prop}\label{prop:projectionYellowTeal}
Let $A\subseteq \N$, $x\in\R^2, e\in\mathcal{S}^1, \ve\in\Q_+$, $C\in\N$ and $t<r\in\R_+$. Suppose that $r$ is sufficiently large and $K^{A,x}_s(e) \geq s - C\log r$, for all $s\leq r-t$. Then the following hold.
\begin{enumerate}
\item If $[t,r]$ is yellow, 
\begin{equation*}
K^A_{r,r,r,t}(x\mid p_e x, e,x) \leq K^A_{r,t}(x\mid x) - (r-t) + \ve r .
\end{equation*}
\item If $[t,r]$ is teal, 
\begin{equation*}
K^A_{r,r,r,t}(x\mid p_e x, e,x) \leq \ve r.
\end{equation*}
\end{enumerate}
\end{prop}

We denote the set of teal intervals by $T$ and the set of yellow intervals by $Y$. 

Supposing that a partition of $[1, r]$ consists of only yellow and teal intervals, and has essentially a constant number of terms, we could repeatedly apply symmetry of information to \ref{prop:projectionYellowTeal} and deduce a useful bound for $K^A_r(x \,|\, p_e x, e)$. We make the notion of such an ``admissible'' partition more precise: a partition $\mathcal{P}=\{[a_i, a_{i+1}]\}_{i=0}^k$ of closed intervals with disjoint interiors is \textbf{\textit{$(M,r,t)$-admissible}} if $[1, r] = \cup_i [a_i, a_{i+1}]$, and it satisfies the following conditions.
\begin{itemize}
\item[\textup{(A1)}] $k \leq M$,
\item[\textup{(A2)}] $[a_i, a_{i+1}]$ is either yellow or teal,
\item[\textup{(A3)}] $a_{i+1} \leq a_i + t$.
\end{itemize}

We can repeatedly apply the symmetry of information to write the complexity $K^A_{r}(x \mid p_e x, e)$  as a sum of complexities over a partition of $[1, r]$, allowing us to apply \ref{prop:projectionYellowTeal}. This idea is encapsulated by the following result in \cite{Stull22c}.

\begin{lem}\label{lem:boundGoodPartitionProjection}
Suppose that $x \in \R^2$, $e \in \mathcal{S}^1$, $\ve\in \Q^+$, $C\in\N$, $t, r \in \N$ satisfy (P1)-(P3). If $\mathcal{P} = \{[a_i, a_{i+1}]\}_{i=0}^k$ is an  $(3C,r,t)$-admissible partition, and $r$ is sufficiently large, then
\begin{align*}
K^A_{r}(x \mid p_e x, e) &\leq \ve r + \sum\limits_{i\in \textbf{Bad}} K^A_{a_{i+1}, a_{i}}(x \mid x) - (a_{i+1} - a_i),
\end{align*}
where
\begin{center}
\textbf{Bad} $=\{i\leq k\mid [a_i, a_{i+1}] \notin T\}$.
\end{center}
\end{lem}

Note that admissible partitions of specified intervals always exist, as per the following lemma:

\begin{lem}\label{lem:goodPartitionProjection}
Let $x\in\R^2$, $r, C\in\N$ and $\frac{r}{C}\leq t < r$. For any $0\leq a < b \leq r$, there is an $(3C,r,t)$-admissible partition of $[a,b]$.
\end{lem}

However, a partition merely being admissible isn't enough to establish the desired bounds. We can do better by consider the special intervals which are both yellow and teal.
\begin{itemize}
\item An interval $[a, b]$ \textbf{\textit{green}} if it is yellow and teal and $b-a\leq t$. 
\end{itemize}
Green intervals are often advantageous because they combine the best of yellow intervals (complexity of $x$ grows superlinearly) and the best of teal intervals (we can compute $x$ with few bits given its projection). We denote the set of green intervals by $G$. We now introduce two more types of intervals to formulate a few results pertaining to green intervals.

\begin{itemize}
\item An interval $[a,b]$ is called \textbf{\textit{red}} if $f$ is strictly increasing on $[a,b]$. 

\item An interval $[a,b]$ is called \textbf{\textit{blue}} if $f$ is constant on $[a,b]$. 
\end{itemize}
In \cite{Stull22c}, a partition $\hat{\mathcal{P}}=\hat{\mathcal{P}}(x, r, t)$ of $[1, r]$ with the following properties is constructed:

\begin{itemize}
\item The interiors of the elements of $\hat{\mathcal{P}}$ are disjoint. 
\item Each interval is red, blue, or green. 
\item If $[a, b]$ is red and $[b, c]$ is blue (not necessarily in $\hat{\mathcal{P}}$), then $b$ is contained in the interior of a green interval in $\hat{\mathcal{P}}$. Moreover, any $b$ that's contained in \emph{any} green interval is contained in a green interval in $\hat{\mathcal{P}}$. 
\item Suppose $I_0,\ldots, I_{n+1}$ is a a red-green-blue sequence in $\hat{\mathcal{P}}$ i.e. a sequence $I_0,\ldots, I_{n+1}$ of consecutive intervals in $\hat{\mathcal{P}}$ such that $I_0$ is red, $I_1,\ldots, I_n$ are green, and $I_{n+1}$ is blue. Then the total length of $I_1,\ldots, I_n$ is at least $t$.

\end{itemize}

Call a maximal collection of consecutive green intervals a \textbf{\textit{green block}}. The last property, that green blocks preceded by a red and succeeded by a blue interval have length at least $t$, will be particularly important. Informally, this property holds because if a green interval had length less than $t$ with red on the left and blue on the right, it would always be possible to lengthen the green by consuming some of the blue and red. 

The final fact from \cite{Stull22c} we need in this section is the following: if there is no red-green-blue sequence in $\hat{\mathcal{P}}= \hat{\mathcal{P}}(x, r, t)$, then there is an essentially all yellow admissible partition of $[1, r]$. More specifically, in this case there is an admissible partition $\mathcal{P}$ such that for some $c$ not depending on $r$, if $c\leq a_i<a_{i+1}$ and $[a_i, a_{i+1}]\in P$, then $[a_i, a_{i+1}]$ is yellow. Intuitively, this is because if a blue interval appears after a red interval, there has to be a red-green-blue sequence somewhere in between. $\dim(x)>1$, so after a certain point there has to be a red interval. Consequently, after a certain point, there can only be red or green intervals, which we can convert into an all-yellow partition of $[c, r]$. 
\begin{remark}
In fact, it is easy to convert a partition of any subinterval of $[1, r]$ consisting of only red and green intervals into an all-yellow $3C$-admissible partition of the subinterval. Just observe that green intervals are yellow, any subinterval of a red interval is yellow, and the union of adjacent yellow intervals is yellow. Greedily combining the red and green intervals from the left to the right and beginning a new yellow interval when the length of the previous yellow is about to exceed $t$ accomplishes this.
\end{remark}

With these definitions and tools, we can now prove the modified projection theorem.

\bigskip

\subsection{Proof of the projection theorem}

\begin{T2}
Let  $x \in \R^2$, $e \in \mathcal{S}^1$, $\ve\in \Q^+$, $C\in\N$, $A\subseteq\mathbb{N}$, and $t, r \in \N$. Suppose that $r$ is sufficiently large, and that the following hold.
\begin{enumerate}
\item[\textup{(P1)}] $1 < d \leq \dim(x) \leq \Dim(x) \leq  D$.
\item[\textup{(P2)}] $ t \geq  \frac{d(2-D)}{2}r$.
\item[\textup{(P3)}] $K^{A, x}_s(e) \geq s - C\log s$, for all $s \leq t$. 
\end{enumerate}
Then 
\begin{equation*}
K^A_r(x \,|\, p_e x, e) \leq \max\{\frac{D-1}{D}(dr - t) + K^A_r(x) - dr, K^A_r(x) - r\} + \ve r.
\end{equation*}
\end{T2}
\begin{proof}
Let $x$ be as in the statement of the theorem, $r$ sufficiently large, and $t \geq \frac{d(2-D)}{2}r$. Let $\hat{P}= \hat{P}(x, r, t)$ be a partition of $[1,r]$ satisfying the properties described in the last section. Let $S$ be the number of red-green-blue sequences in $\hat{P}$. Note that $S<\frac{2}{d(2-D)}$, since the green block in each red-green-blue sequence is at least length $t$, and these blocks have to be separated from each other by some amount.

\vspace{2mm}

\noindent \textbf{The case $S$ = 0:}
In this case, there are no red-green-blue sequences, so let $\mathcal{P}$ be the all yellow $3C$-admissible partition guaranteed by last fact of the previous section. 

\begin{equation*}
\sum\limits_{I_i \in \mathcal{P}-Y} a_{i+1} - a_i \leq c,
\end{equation*}
for some constant $c$, and so for sufficiently large $r$
\begin{equation}
B := \sum\limits_{I_i \in \mathcal{P}\cap Y} a_{i+1} - a_i \geq r - \frac{\ve r}{2}.
\end{equation}
By symmetry of information, we can write
\begin{align*}
K^A_r(x) &\geq \sum\limits_{I_i \in \mathcal{P}} K^A_{a_{i+1}, a_i}(x\mid x) - O(\log r)\tag*{}\\
&\geq \sum\limits_{I_i \in \mathcal{P}\cap Y} K^A_{a_{i+1}, a_i}(x\mid x) - O(\log r)\\
&\geq -\frac{\ve r}{4} + \sum\limits_{I_i \in \mathcal{P}\cap Y} K^A_{a_{i+1}, a_i}(x\mid x)\tag*{}
\end{align*}

To apply Lemma \ref{lem:boundGoodPartitionProjection} we note that, on green intervals, $K^A_{a, b}(x\mid x) = b - a$. Therefore, we see that
\begin{align*}
K^A_r(x) &\geq K^A_r(x\mid p_e x, e) + B -\frac{\ve r}{2}\tag*{[Lemma \ref{lem:boundGoodPartitionProjection}]}\\
&\geq K^A_r(x\mid p_e, x, e) + r - \ve r.
\end{align*}
Thus, 
\begin{equation}\label{sZero1}
K^A_r(x\mid p_e x, e) \leq K^A_r(x) - r + \ve r,
\end{equation}
and the conclusion follows.

\vspace{2mm}

\noindent \textbf{The case $S$ = 1:} Now, suppose there is exactly one red-green-blue sequence in $\hat{\mathcal{P}}$. Then there is a precision $1 < r_1 < r - t$ and an $s\geq t$ such that $[r_1, r_1+s]$ is green in $\hat{\mathcal{P}}$. Let $\mathcal{P}_1$ be a $3C$-admissible partition of $[1,r_1]$, and $\mathcal{P}_2$ be a $3C$-admissible partition of $[r_1 +s, r]$, guaranteed by Lemma \ref{lem:goodPartitionProjection}. Since there is exactly one red-green-blue sequence in $\hat{\mathcal{P}}$, $\mathcal{P}_1$ contains no red-green-blue sequences. Therefore, using the same argument as in the previous case, $\mathcal{P}_1$ is essentially covered by yellow intervals and we conclude that 
\begin{align*}
    K^A_{r_1}(x \mid p_e x, e) &\leq K^A_{r_1}(x) - r_1 + \frac{\ve r_1}{4}\\
    (D-1)r_1 + \frac{\ve r_1}{2}.
\end{align*}
Let $r_2 \geq r_1+s$ be minimal precision such that $[r_2, r]$ is the union of yellow intervals\footnote{We allow $r_2$ to be equal to $r$, in the case that $[r_1+s,r]$ is covered by teal intervals.} Therefore we have
\begin{align*}
    K^A_{r}(x\mid p_e x, e) &\leq K^A_{r_1}(x) - r_1 + K^A_{r, r_2}(x\mid x) - (r-r_2) + \ve r\\
    &\leq \left(D - 1\right)r_1 + K^A_{r, r_2}(x\mid x) - (r-r_2) + \ve r\\
    &\leq \left(D - 1\right)r_1 + \left(d - 1\right)\left(r-r_2\right) + K^A_r(x) - d r + \ve r\\
    &\leq \left(D - 1\right)B + K^A_r(x) - d r + \ve r.
\end{align*}
If $B \leq \frac{d r - t}{D}$, the conclusion follows. So, we assume that $B > \frac{d r - t}{D}$. Hence,
\begin{align*}
    K^A_{r}(x\mid p_e x, e) &\leq K^A_r(x) - s - B+ \ve r\\
    &\leq K^A_r(x) - t - B+ \ve r\\
    &< K^A_r(x) - t - \frac{d r - t}{D}+ \ve r\\
    &= K^A_r(x) - d r + \frac{D-1}{D}\left(d r - t\right)+ \ve r,
\end{align*}
and the conclusion follows.

\noindent \textbf{The case $S\geq 2$:}
We now consider the case that there are at least two red-green-blue sequence in $\hat{P}$. Let 
\begin{equation}\label{eq:lengthOfTealIntervals}
L = \sum\limits_{I_i \in \mathcal{P}\cap G} a_{i+1} - a_i
\end{equation}
be the total length of the green intervals in $\mathcal{P}$. In this case we have $L \geq 2t$. Let 
\begin{equation}
B = \sum\limits_{i\in \textbf{Bad}} a_{i+1}-a_i
\end{equation}
be the total length of the bad (non-teal) intervals in $\mathcal{P}$. 

We first prove that 
\begin{equation}\label{eq:projectionMainThm1}
K^A_r(x\mid p_e x, e) \leq \min\{K^A_r(x) - B - 2t, B\} +\ve r.
\end{equation}
Since $x$ is an element of $\R^2$,
\begin{equation*}
K^A_{a_{i+1}, a_i}(x\mid x) \leq 2(a_{i+1} - a_i) + O(\log r).
\end{equation*}
Therefore, by Lemma \ref{lem:boundGoodPartitionProjection}, with respect to $\ve / 4$,
\begin{equation}\label{eq:projectionMainThm2}
K^A_r(x\mid p_e x, e) \leq \frac{\ve r}{2} + B.
\end{equation}
By repeated applications of the symmetry of information,
\begin{align}
K^A_r(x) &\geq -\frac{\ve r}{2} + \sum\limits_{I_i \in \mathcal{P}\cap T} K^A_{a_{i+1}, a_i}(x\mid x) + \sum\limits_{i\in \textbf{Bad}} K^A_{a_{i+1}, a_i}(x\mid x)\tag*{}\\
&\geq -\frac{\ve r}{2} + \sum\limits_{I_i \in \mathcal{P}\cap G} K^A_{a_{i+1}, a_i}(x\mid x) + \sum\limits_{i\in \textbf{Bad}} K^A_{a_{i+1}, a_i}(x\mid x)\tag*{}\\
&= -\frac{\ve r}{2} +L + \sum\limits_{i\in \textbf{Bad}} K^A_{a_{i+1}, a_i}(x\mid x)\tag*{}\\
&\geq 2t + K^A_r(x\mid p_e x, e) + B -\ve r\label{eq:projectionMainThm3}
\end{align}
Combining (\ref{eq:projectionMainThm2}) and (\ref{eq:projectionMainThm3}) proves inequality (\ref{eq:projectionMainThm1}). 

By inequality (\ref{eq:projectionMainThm1}), if
\begin{equation*}
B \leq K^A_r(x) - d r+ \frac{D-1}{D}(dr - t),
\end{equation*}
we are done, so we assume otherwise. Applying (\ref{eq:projectionMainThm1}) again and using our assumption on $t$ implies that
\begin{align*}
K^A_r(x\mid p_e x,e) &\leq K^A_r(x) - 2t - B+\ve r\\
&< \frac{d}{D}r - \frac{D+1}{D}t+\ve r\\
&\leq \frac{D-1}{D}(dr-t) +\ve r\\
&\leq K^A_r(x) - d r + \frac{D-1}{D}\left(d r - t\right) + \ve r,
\end{align*}
and the proof is complete.
\end{proof}

\bigskip

\section{Effective dimension of distances}\label{sec:effdim}

In the previous section, we considered partitions of the interval $[1, r]$ depending on the complexity function of our pinned point $x$. In particular, we were able to use these partitions to get a bound on the complexity of $x$ given $e$ and the projection of $x$ in the direction of $e$. Now, we need to consider the complexity function of $y$ and relate this to the complexity of $\vert x-y\vert$. Similar to before, we let $f:\R_+ \rightarrow \R_+$ be the piece-wise linear function which agrees with $K^A_s(y)$ on the integers, and 
\begin{center}
$f(a) = f(\lfloor a\rfloor) + (a -\lfloor a\rfloor)(f(\lceil a \rceil) - f(\lfloor a\rfloor)) $,
\end{center}
for any non-integer $a$. Since
\begin{center}
$K^A_s(y) \leq K^A_{s+1}(y)$ 
\end{center} 
for every $s\in\N$, $f$ is non-decreasing and since $y\in \R^2$, for every $a < b\in\R$,
\begin{center}
$f(b) - f(a) \leq 2(b-a) + O(\log \lceil b \rceil)$.
\end{center}

As before, we make the following definitions: an interval $[a,b]$ is called \textbf{\textit{teal}} if $f(b) - f(c) \leq b-c$ for every $c\in[a,b]$. It is called \textbf{\textit{yellow}} if $f(c) - f(a) \geq c - a$, for every $c\in[a,b]$. We denote the set of teal intervals by $T$ and the set of yellow intervals by $Y$. 

For reference, here we list some conditions that our points $x$ and $y$ will be assumed to satisfy throughout the remainder of this section. Let $x,y\in\R^2$, $e = \frac{x-y}{\vert x - y\vert}$ and $A, B \subseteq \N$. For this section, we let $d_x = \dim^A(x)$, $D_x = \Dim^A(x)$, $d_y = \dim^A(y)$ and $D_y = \Dim^A(y)$. We will assume that $x$ and $y$ satisfy the following conditions.
\begin{itemize}
\item[\textup{(C1)}] $d_x,d_y > 1$
\item[\textup{(C2)}] $K^{x,A}_r(e) = r - O(\log r)$ for all $r$.
\item[\textup{(C3)}] $K^{x,A, B}_r(y) \geq K^{A}_r(y) - O(\log r)$ for all sufficiently large $r$. 
\item[\textup{(C4)}] $K^{A}_r(e\mid y) = r - o(r)$ for all $r$.
\end{itemize}

\bigskip

\subsection{Complexity of distances on yellow and teal intervals}\label{subsec:distYellowTeal}

As compared to the corresponding part of \cite{Stull22c}, we'll need to partition $[1, r]$ more carefully. However, similar to the proof of the projection theorem in the previous section, there are a few tools from \cite{Stull22c} that we can reuse. 

\begin{lem}\label{lem:distancesYellowTeal}
Suppose that $A\subseteq\N$, $x, y \in \R^2$ and $e = \frac{x - y}{\vert x - y\vert}$ satisfy (C1)-(C4) for every $r\in\N$. Then for every $\ve\in\Q_+$ and all sufficiently large $r\in\N$, the following hold.
\begin{enumerate}
\item If $[t,r]$ is yellow and $t\leq r\leq 2t$
\begin{equation*}
K^{A,x}_{r,r,t}(y\mid \vert x-y\vert, y) \leq K^{A}_{r,t}(y\mid y) - (r-t) + \ve r .
\end{equation*}
\item If $[t,r]$ is teal,  and $t\leq r\leq 2t$,
\begin{equation*}
K^{A,x}_{r,r,t}(y\mid \vert x-y\vert, y) \leq \ve r .
\end{equation*}
\end{enumerate}
\end{lem}

We say that a partition $\mathcal{P} = \{[a_i, a_{i+1}]\}_{i=0}^k$ of intervals with disjoint interiors is \textbf{\textit{good}} if $[1, r] = \cup_i [a_i, a_{i+1}]$ and it satisfies the following conditions.
\begin{itemize}
\item[\textup{(G1)}] $[a_i, a_{i+1}]$ is either yellow or teal,
\item[\textup{(G2)}] $a_{i+1} \leq 2a_i$, for every $i$ and
\item[\textup{(G3)}] $a_{i+2} > 2 a_{i}$ for every $i < k$.
\end{itemize}
Note that (G3) ensures that the errors do not pile up. Furthermore, observe that (G2) is somewhat different than the admissibility condition in the previous section. There, we had that an interval could not be longer than $t$: essentially some fixed quantity, at least when partitioning. Now, the requirement is that the intervals cannot be any more than "doubling", so the best we can hope for in a partition is a logarithmic number terms. Indeed, just as we had admissible partitions for every $[a, b]$, the following lemma guarantees the existence of good partitions. 

\begin{lem}\label{lem:existenceOfGoodPartitionDistance}
For every $y\in \R^2$ and every $r\in\N$, there is a good partition of $[1, r]$.
\end{lem}

The following lemma uses repeated applications of the symmetry of information to write $K^x_r(y\mid \vert x-y\vert)$ as a sum of its complexity on the intervals of a partition. The conclusion then follows via Lemma \ref{lem:distancesYellowTeal}. 
\begin{lem}\label{lem:boundGoodPartitionDistance}
Let $A\subseteq\N$. Let $\mathcal{P} = \{[a_i, a_{i+1}]\}_{i=0}^k$ be a good partition. Then
\begin{align*}
K^{A,x}_{r}(y \mid\vert x-y\vert) &\leq \ve r + \sum\limits_{i\in \textbf{Bad}} K^{A}_{a_{i+1}, a_{i}}(y \mid y) - (a_{i+1} - a_i),
\end{align*}
where
\begin{center}
\textbf{Bad} $=\{i\leq k\mid [a_i, a_{i+1}] \notin T\}$.
\end{center}
\end{lem}

\noindent So applying the previous lemma with $\frac{\ve}{2}$, absorbing the log term for sufficiently large $r$, and recalling condition (C3), we have that

\begin{equation}\label{eq:distancepartitionbound}
    K^{x, A}_r(\vert x - y\vert) \geq K^A_r(y) - \sum\limits_{i\in \textbf{Bad}} K^A_{a_{i+1}, a_{i}}(y \mid y) - (a_{i+1} - a_i) - \ve r,
\end{equation}

\noindent Constructing a particular good partition to optimize this bound - either at every precision or well-chosen precisions - will be crucial to proving a bound on the effective Hausdorff and effective packing dimension of such points (respectively) and will thus be a key goal of the remainder of this section. 

\bigskip

\subsection{Sufficient conditions for an all-yellow partition}

We will describe a more general partition strategy in the next subsection, but here we will introduce a related partition for a simpler scenario. In particular, we will discuss situations where we can guarantee the existence of a good partition consisting (essentially) of only yellow intervals which are not more than doubling. To see why this is significant, observe that if we had such a partition $\mathcal{P}$, equation $\ref{eq:distancepartitionbound}$, together with the observation that the complexity of $y$ grows at an average rate of exactly 1 on yellow intervals that are \emph{also} green, implies that for sufficiently large $r$

\begin{align*}
    K^{x, A}_r(\vert x - y\vert) &\geq K^A_r(y) - \sum\limits_{i\in \mathcal{P}} K^A_{a_{i+1}, a_{i}}(y \mid y) - (a_{i+1} - a_i) - \frac{\ve}{2} r\\
    &\geq K^A_r(y) - K^A_r(y) + r - \ve r\\
    &= (1-\ve) r
\end{align*}

\noindent Since we can take $\ve$ to be as small as desired, this implies the dimension of the distance is 1. We now state the necessary conditions and make the above argument more rigorous. 

\begin{prop}\label{prop:yellowfulldimension}
Suppose that, in addition to conditions $(C1)-(C4)$, we also have that $D_y<2d_y-1$. Then $\dim^{x, A}(\vert x-y\vert)=1$.
\end{prop}

Intuitively, this extra requirement expresses that, if $y$ is semi-regular, then we have the same conclusion as if $y$ is regular (in the sense of having equal effective Hausdorff and effective packing dimension). As long as these dimensions do not differ by too much, we can find an all-yellow partition. To see why we have this bound specifically, consider the adversary complexity function in Figure \ref{fig:allyellow}. We can only work with yellow intervals that are at most doubling, so we want to choose $D_y$ and $d_y$ to guarantee that the average growth rate of the complexity from $\frac{r}{2}$ to $r$ is at least 1. Maximizing $K^A_{\frac{r}{2}}(y)$ and minimizing $K^A_{r}(y)$ allows us to conclude the above when $D_y<2d_y-1$. Now, we prove the proposition.

\begin{proofof}{Proposition \ref{prop:yellowfulldimension}}
Pick an $\ve>0$ small enough that $d_x-\frac{\ve}{4}>1$ and $D_y+\frac{\ve}{4}<2(d_y-\frac{\ve}{4})-1$. For all $r$ sufficiently large, we have that $(d_y - \frac{\ve}{4}) s < K^A_s(y)<(D_y + \frac{\ve}{4})s$ whenever $s\geq\frac{\ve}{4} r$. We now construct a partition of the interval $[\frac{\ve}{2} r, r]$.

Set $r_0 = r$. Given $r_i$, first check if $r_i\leq \frac{\ve}{2}r$. If it is, let $i=n$ and we are done. If not, we let $r_{i+1}$ be the largest value of $s$ such that $K^A_s(y) = K^A_{\frac{r_i}{2}}(y) + (s-\frac{r_i}{2})$. Note that $\frac{r_i}{2}>\frac{\ve}{4}r$, so the inequalities at the start of the proof hold in the interval we consider on this step. First, we show that $r_{i+1}<r_i$. To see this, observe that  

\begin{equation*}
K^A_{\frac{r_i}{2}}(y) + (s-\frac{r_i}{2})< s+(D_y+\frac{\ve}{4}-\frac{1}{2})r_i
\end{equation*}

\noindent On the other hand, 

\begin{equation*}
K^A_s(y)>(d_y -\frac{\ve}{4})s,
\end{equation*}

\noindent so for an $s$ satisfying the above equation, we have

\begin{equation*}
    s<\dfrac{\frac{D_y}{2} - \frac{1}{2} + \frac{\ve}{8}}{d_y-1-\frac{\ve}{4}}r_i<r_i. 
\end{equation*}

\noindent Here, the second inequality follows from our choice of $\ve$. Now, we define the partition $\mathcal{P}$ to be $ [1, r_n], [r_n, r_{n-1}], ..., [r_1, r_0]$. 

We now claim that each $[r_{i+1}, r_i]$ is a yellow interval which is at most doubling. If it were not yellow, we would have some $s^\prime\in[r_{i+1}, r_i]$ such that $K^A_{s^\prime}(y)-K^A_{r_{i+1}}(y)<s^\prime - r_{i + 1}$. This implies that $K^A_{s^\prime}(y) < K^A_{\frac{r_i}{2}}(y) + (s^\prime-\frac{r_i}{2})$. However, $K^A_{r_i}(y) > K^A_{\frac{r_i}{2}}(y) + (r_i-\frac{r_i}{2})$, so by the intermediate value theorem, there is some $s^{\prime\prime}\in(s^\prime, r_i)$ such that $K^A_{s^{\prime\prime}}(y) = K^A_{\frac{r_i}{2}}(y) + (s^{\prime\prime}-\frac{r_i}{2})$, contradicting that $r_{i+1}$ was the maximal such precision. Finally, $[r_{i + 1}, r_i]$ is at most doubling because $r_{i+1}$ cannot be any less than $\frac{r_i}{2}$, as $K^A_{\frac{r_i}{2}}(y) = K^A_{\frac{r_i}{2}}(y) + (\frac{r_i}{2}-\frac{r_i}{2})$. Thus we have a partition of $[1, r]$ where all but the first interval are yellow.

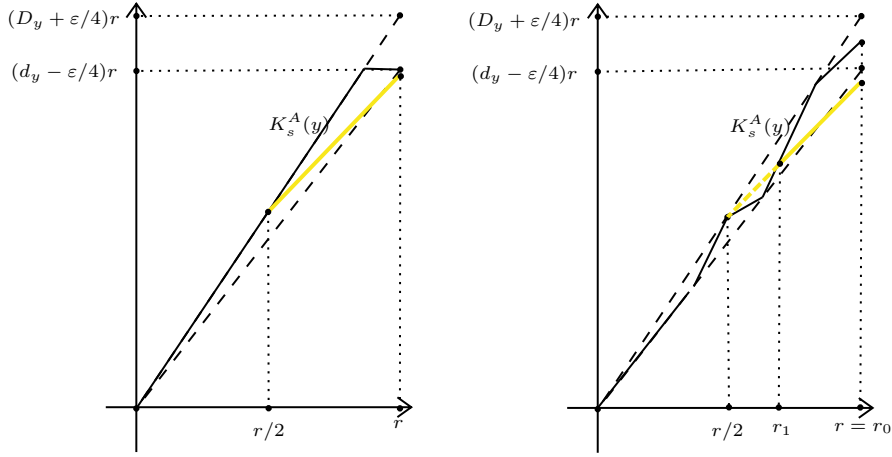
\begin{figure}[t]

\centering

\tikzset{every picture/.style={line width=0.75pt}} 

\begin{tikzpicture}[x=0.75pt,y=0.75pt,yscale=-1,xscale=1]

\draw  (-41.89,310.34) -- (112.19,310.34)(-26.48,106.67) -- (-26.48,332.97) (105.19,305.34) -- (112.19,310.34) -- (105.19,315.34) (-31.48,113.67) -- (-26.48,106.67) -- (-21.48,113.67)  ;
\draw  [fill={rgb, 255:red, 0; green, 0; blue, 0 }  ,fill opacity=1 ] (-27.56,310.97) .. controls (-27.56,310.35) and (-27.08,309.86) .. (-26.48,309.86) .. controls (-25.89,309.86) and (-25.41,310.35) .. (-25.41,310.97) .. controls (-25.41,311.58) and (-25.89,312.08) .. (-26.48,312.08) .. controls (-27.08,312.08) and (-27.56,311.58) .. (-27.56,310.97) -- cycle ;
\draw  [fill={rgb, 255:red, 0; green, 0; blue, 0 }  ,fill opacity=1 ] (105.57,139.92) .. controls (105.57,139.31) and (106.05,138.81) .. (106.64,138.81) .. controls (107.24,138.81) and (107.72,139.31) .. (107.72,139.92) .. controls (107.72,140.53) and (107.24,141.03) .. (106.64,141.03) .. controls (106.05,141.03) and (105.57,140.53) .. (105.57,139.92) -- cycle ;
\draw  [fill={rgb, 255:red, 0; green, 0; blue, 0 }  ,fill opacity=1 ] (105.4,112.58) .. controls (105.4,111.97) and (105.88,111.47) .. (106.48,111.47) .. controls (107.07,111.47) and (107.56,111.97) .. (107.56,112.58) .. controls (107.56,113.2) and (107.07,113.69) .. (106.48,113.69) .. controls (105.88,113.69) and (105.4,113.2) .. (105.4,112.58) -- cycle ;
\draw  [dash pattern={on 4.5pt off 4.5pt}]  (106.48,112.58) -- (-26.48,310.97) ;
\draw  [dash pattern={on 4.5pt off 4.5pt}]  (106.64,139.92) -- (-26.48,310.97) ;
\draw  [fill={rgb, 255:red, 0; green, 0; blue, 0 }  ,fill opacity=1 ] (105.4,310.94) .. controls (105.39,310.32) and (105.87,309.82) .. (106.46,309.81) .. controls (107.05,309.8) and (107.55,310.28) .. (107.56,310.9) .. controls (107.57,311.51) and (107.09,312.02) .. (106.5,312.03) .. controls (105.9,312.04) and (105.41,311.55) .. (105.4,310.94) -- cycle ;
\draw  [fill={rgb, 255:red, 0; green, 0; blue, 0 }  ,fill opacity=1 ] (39.17,310.87) .. controls (39.16,310.26) and (39.63,309.75) .. (40.23,309.74) .. controls (40.82,309.73) and (41.31,310.22) .. (41.32,310.83) .. controls (41.33,311.44) and (40.86,311.95) .. (40.26,311.96) .. controls (39.67,311.97) and (39.18,311.48) .. (39.17,310.87) -- cycle ;
\draw  [fill={rgb, 255:red, 0; green, 0; blue, 0 }  ,fill opacity=1 ] (-27.51,140.64) .. controls (-27.52,140.03) and (-27.04,139.52) .. (-26.45,139.51) .. controls (-25.85,139.5) and (-25.36,139.99) .. (-25.35,140.6) .. controls (-25.34,141.22) and (-25.82,141.72) .. (-26.41,141.73) .. controls (-27.01,141.74) and (-27.5,141.26) .. (-27.51,140.64) -- cycle ;
\draw  [fill={rgb, 255:red, 0; green, 0; blue, 0 }  ,fill opacity=1 ] (-27.51,112.88) .. controls (-27.52,112.27) and (-27.04,111.76) .. (-26.45,111.75) .. controls (-25.85,111.74) and (-25.36,112.23) .. (-25.35,112.84) .. controls (-25.34,113.45) and (-25.82,113.96) .. (-26.41,113.97) .. controls (-27.01,113.98) and (-27.5,113.49) .. (-27.51,112.88) -- cycle ;
\draw  (190.81,310.7) -- (344.89,310.7)(206.22,107.03) -- (206.22,333.33) (337.89,305.7) -- (344.89,310.7) -- (337.89,315.7) (201.22,114.03) -- (206.22,107.03) -- (211.22,114.03)  ;
\draw  [fill={rgb, 255:red, 0; green, 0; blue, 0 }  ,fill opacity=1 ] (205.14,311.33) .. controls (205.14,310.72) and (205.62,310.22) .. (206.22,310.22) .. controls (206.81,310.22) and (207.29,310.72) .. (207.29,311.33) .. controls (207.29,311.94) and (206.81,312.44) .. (206.22,312.44) .. controls (205.62,312.44) and (205.14,311.94) .. (205.14,311.33) -- cycle ;
\draw  [fill={rgb, 255:red, 0; green, 0; blue, 0 }  ,fill opacity=1 ] (338.27,139.17) .. controls (338.27,138.56) and (338.75,138.06) .. (339.35,138.06) .. controls (339.94,138.06) and (340.42,138.56) .. (340.42,139.17) .. controls (340.42,139.79) and (339.94,140.28) .. (339.35,140.28) .. controls (338.75,140.28) and (338.27,139.79) .. (338.27,139.17) -- cycle ;
\draw  [fill={rgb, 255:red, 0; green, 0; blue, 0 }  ,fill opacity=1 ] (338.1,112.94) .. controls (338.1,112.33) and (338.59,111.83) .. (339.18,111.83) .. controls (339.78,111.83) and (340.26,112.33) .. (340.26,112.94) .. controls (340.26,113.56) and (339.78,114.05) .. (339.18,114.05) .. controls (338.59,114.05) and (338.1,113.56) .. (338.1,112.94) -- cycle ;
\draw  [dash pattern={on 4.5pt off 4.5pt}]  (339.18,112.94) -- (206.22,311.33) ;
\draw  [dash pattern={on 4.5pt off 4.5pt}]  (339.35,140.28) -- (206.22,311.33) ;
\draw  [fill={rgb, 255:red, 0; green, 0; blue, 0 }  ,fill opacity=1 ] (337.57,310.48) .. controls (337.56,309.87) and (338.04,309.36) .. (338.63,309.35) .. controls (339.23,309.34) and (339.72,309.83) .. (339.73,310.44) .. controls (339.74,311.05) and (339.26,311.56) .. (338.67,311.57) .. controls (338.07,311.58) and (337.58,311.09) .. (337.57,310.48) -- cycle ;
\draw  [fill={rgb, 255:red, 0; green, 0; blue, 0 }  ,fill opacity=1 ] (271.34,310.41) .. controls (271.33,309.8) and (271.8,309.29) .. (272.4,309.28) .. controls (272.99,309.27) and (273.48,309.76) .. (273.49,310.37) .. controls (273.5,310.99) and (273.03,311.49) .. (272.44,311.5) .. controls (271.84,311.51) and (271.35,311.02) .. (271.34,310.41) -- cycle ;
\draw  [fill={rgb, 255:red, 0; green, 0; blue, 0 }  ,fill opacity=1 ] (205.19,141) .. controls (205.18,140.39) and (205.66,139.89) .. (206.25,139.87) .. controls (206.85,139.86) and (207.34,140.35) .. (207.35,140.96) .. controls (207.36,141.58) and (206.89,142.08) .. (206.29,142.09) .. controls (205.7,142.1) and (205.21,141.62) .. (205.19,141) -- cycle ;
\draw  [fill={rgb, 255:red, 0; green, 0; blue, 0 }  ,fill opacity=1 ] (205.19,113.24) .. controls (205.18,112.63) and (205.66,112.12) .. (206.25,112.11) .. controls (206.85,112.1) and (207.34,112.59) .. (207.35,113.2) .. controls (207.36,113.81) and (206.89,114.32) .. (206.29,114.33) .. controls (205.7,114.34) and (205.21,113.85) .. (205.19,113.24) -- cycle ;
\draw    (-26.48,310.97) -- (88.43,139.44) ;
\draw    (88.43,139.44) -- (106.64,139.92) ;
\draw    (252.29,252.13) -- (206.22,310.7) ;
\draw [color={rgb, 255:red, 248; green, 231; blue, 28 }  ,draw opacity=1 ][line width=1.5]    (40,211.78) -- (105.82,143.31) ;
\draw  [fill={rgb, 255:red, 0; green, 0; blue, 0 }  ,fill opacity=1 ] (38.92,211.78) .. controls (38.92,211.16) and (39.4,210.67) .. (40,210.67) .. controls (40.59,210.67) and (41.07,211.16) .. (41.07,211.78) .. controls (41.07,212.39) and (40.59,212.89) .. (40,212.89) .. controls (39.4,212.89) and (38.92,212.39) .. (38.92,211.78) -- cycle ;
\draw    (254.68,249.4) -- (271.63,214.45) ;
\draw    (271.63,214.45) -- (289.38,204.34) ;
\draw    (289.38,204.34) -- (316.14,147.27) ;
\draw    (316.14,147.27) -- (338.4,126.24) ;
\draw  [fill={rgb, 255:red, 0; green, 0; blue, 0 }  ,fill opacity=1 ] (270.56,214.45) .. controls (270.56,213.83) and (271.04,213.34) .. (271.63,213.34) .. controls (272.23,213.34) and (272.71,213.83) .. (272.71,214.45) .. controls (272.71,215.06) and (272.23,215.56) .. (271.63,215.56) .. controls (271.04,215.56) and (270.56,215.06) .. (270.56,214.45) -- cycle ;
\draw [color={rgb, 255:red, 248; green, 231; blue, 28 }  ,draw opacity=1 ][line width=1.5]  [dash pattern={on 3.75pt off 1.5pt}]  (298.13,187.41) -- (271.63,214.45) ;
\draw [color={rgb, 255:red, 248; green, 231; blue, 28 }  ,draw opacity=1 ][line width=1.5]    (298.13,187.41) -- (338.4,146.72) ;
\draw  [fill={rgb, 255:red, 0; green, 0; blue, 0 }  ,fill opacity=1 ] (297.05,187.41) .. controls (297.05,186.8) and (297.53,186.3) .. (298.13,186.3) .. controls (298.72,186.3) and (299.2,186.8) .. (299.2,187.41) .. controls (299.2,188.02) and (298.72,188.52) .. (298.13,188.52) .. controls (297.53,188.52) and (297.05,188.02) .. (297.05,187.41) -- cycle ;
\draw  [dash pattern={on 0.84pt off 2.51pt}]  (298.13,187.41) -- (297.54,310.48) ;
\draw  [dash pattern={on 0.84pt off 2.51pt}]  (271.63,214.45) -- (272.42,310.39) ;
\draw  [fill={rgb, 255:red, 0; green, 0; blue, 0 }  ,fill opacity=1 ] (296.46,310.5) .. controls (296.45,309.88) and (296.93,309.38) .. (297.52,309.37) .. controls (298.12,309.36) and (298.61,309.85) .. (298.62,310.46) .. controls (298.63,311.07) and (298.15,311.58) .. (297.56,311.59) .. controls (296.96,311.6) and (296.47,311.11) .. (296.46,310.5) -- cycle ;
\draw  [dash pattern={on 0.84pt off 2.51pt}]  (40,211.78) -- (40.25,310.85) ;
\draw  [fill={rgb, 255:red, 0; green, 0; blue, 0 }  ,fill opacity=1 ] (338.4,126.24) .. controls (338.4,125.63) and (338.88,125.13) .. (339.47,125.13) .. controls (340.07,125.13) and (340.55,125.63) .. (340.55,126.24) .. controls (340.55,126.85) and (340.07,127.35) .. (339.47,127.35) .. controls (338.88,127.35) and (338.4,126.85) .. (338.4,126.24) -- cycle ;
\draw  [fill={rgb, 255:red, 0; green, 0; blue, 0 }  ,fill opacity=1 ] (338.4,146.72) .. controls (338.4,146.11) and (338.88,145.61) .. (339.47,145.61) .. controls (340.07,145.61) and (340.55,146.11) .. (340.55,146.72) .. controls (340.55,147.33) and (340.07,147.83) .. (339.47,147.83) .. controls (338.88,147.83) and (338.4,147.33) .. (338.4,146.72) -- cycle ;
\draw  [fill={rgb, 255:red, 0; green, 0; blue, 0 }  ,fill opacity=1 ] (105.82,143.31) .. controls (105.82,142.69) and (106.3,142.2) .. (106.89,142.2) .. controls (107.49,142.2) and (107.97,142.69) .. (107.97,143.31) .. controls (107.97,143.92) and (107.49,144.42) .. (106.89,144.42) .. controls (106.3,144.42) and (105.82,143.92) .. (105.82,143.31) -- cycle ;
\draw  [dash pattern={on 0.84pt off 2.51pt}]  (106.64,139.92) -- (106.46,309.81) ;
\draw  [dash pattern={on 0.84pt off 2.51pt}]  (339.47,127.35) -- (338.67,311.57) ;
\draw  [dash pattern={on 0.84pt off 2.51pt}]  (106.64,139.92) -- (-26.43,140.62) ;
\draw  [dash pattern={on 0.84pt off 2.51pt}]  (106.48,112.58) -- (-27.51,112.88) ;
\draw  [dash pattern={on 0.84pt off 2.51pt}]  (206.27,140.98) -- (339.35,139.17) ;
\draw  [dash pattern={on 0.84pt off 2.51pt}]  (206.27,113.22) -- (339.18,112.94) ;

\draw (38.67,160.52) node [anchor=north west][inner sep=0.75pt]  [font=\scriptsize]  {$K^A_{s}( y)$};
\draw (101.52,315.39) node [anchor=north west][inner sep=0.75pt]  [font=\scriptsize]  {$r$};
\draw (30.82,316.41) node [anchor=north west][inner sep=0.75pt]  [font=\scriptsize]  {$r/2$};
\draw (-91.2,135.29) node [anchor=north west][inner sep=0.75pt]  [font=\scriptsize]  {$( d_{y} -\varepsilon /4) r$};
\draw (-93.09,108.54) node [anchor=north west][inner sep=0.75pt]  [font=\scriptsize]  {$( D_{y} +\varepsilon /4) r$};
\draw (271.18,161.22) node [anchor=north west][inner sep=0.75pt]  [font=\scriptsize]  {$K^A_{s}( y)$};
\draw (323.91,316.6) node [anchor=north west][inner sep=0.75pt]  [font=\scriptsize]  {$r=r_{0}$};
\draw (262.18,316.29) node [anchor=north west][inner sep=0.75pt]  [font=\scriptsize]  {$r/2$};
\draw (140.86,136.09) node [anchor=north west][inner sep=0.75pt]  [font=\scriptsize]  {$( d_{y} -\varepsilon /4) r$};
\draw (139.17,109.46) node [anchor=north west][inner sep=0.75pt]  [font=\scriptsize]  {$( D_{y} +\varepsilon /4) r$};
\draw (292.67,317.32) node [anchor=north west][inner sep=0.75pt]  [font=\scriptsize]  {$r_{1}$};

\end{tikzpicture}
\caption{On the left, the adversary complexity function that grows as quickly as possible (given $D_y$) and then levels off. On the right, an implementation of the procedure, generating a yellow interval $[r_1, r]$ by sending out a line of slope 1 from $\left(\frac{r}{2}, K^A_{\frac{r}{2}}(y)\right)$ and then finding the last intersection of this line with $K^A_s(y)$.}
\label{fig:allyellow}
\end{figure}

We want to apply Lemma \ref{lem:boundGoodPartitionDistance}, so we make $\mathcal{P}$ into a good partition by taking a good partition $\mathcal{P}_{[1, r_n]}$ of the first interval $[1, r_n]$ and replacing it in $\mathcal{P}$ with $\mathcal{P}_{[1, r_n]}$. As for the remaining intervals, the union of yellow intervals is still yellow, so simply greedily combine them from the left to the right. Start a new yellow interval each time adding the next $[r_{i+1}, r_i]$ would make the current yellow interval more than doubling. For ease of notation, continue to denote this modification of  $\mathcal{P}$ by $\mathcal{P}$. The conditions being satisfied, we apply Lemma \ref{lem:boundGoodPartitionDistance} relative to $A$ with $\frac{\ve}{2}$ via (\ref{eq:distancepartitionbound}), obtaining:

\begin{align*}
    K^{x, A}_r(\vert x - y\vert) &\geq K^A_r(y)  - \frac{\ve}{2} r \sum\limits_{i\in \mathcal{P}\setminus\mathcal{P}_{[1, r_n]}} K^A_{a_{i+1}, a_{i}}(y \mid y) - (a_{i+1} - a_i)\\
    &\geq K^A_r(y) - (K^A_r(y)-r + \frac{\ve}{2}r) - \frac{\ve}{2}r\\
    &=(1 - \ve)r
\end{align*}

\noindent Since we only needed $r$ to be sufficiently large, 

\begin{equation*}
\dim^{x, A}(\vert x-y\vert)\geq 1-\ve.
\end{equation*}

\noindent and taking a sequence of $\ve$ going to 0 gives the desired conclusion.

\end{proofof}

\bigskip

\subsection{Constructing a general partition}

In this subsection, we describe a partitioning strategy that works outside of the special case considered in Proposition \ref{prop:yellowfulldimension}. The key limitation of the previous subsection was that we could only use intervals that were at most doubling, which we now would like to enhance - at least for certain intervals - using the projection theorem of section 3. The new partition will involve a combination of yellow intervals and certain teal intervals, with the teal intervals chosen so we can apply Theorem \ref{thm:modifiedProjectionTheorem} to understand the complexity growth of $\vert x-y\vert$ on them. To begin, fix a precision $r\in\N$. To keep the expressions involved reasonably brief, we set
\begin{center}
    $d = \min\{d_x, d_y\}$ and $D = \max\{D_x, D_y\}$.\footnote{It would be possible to obtain a somewhat better bound in Theorem \ref{thm:moregeneralmaintheorem} that more freely involves $d_x, d_y, D_x,$ and $D_y$ using the same approach of this section, at the cost of significantly worse calculations.}
\end{center}
Our goal is to give a lower bound on the complexity $K^{A,x}_r(\vert x-y\vert)$ at precision $r$. We will first define a sequence $r = r_0 > r_1 >\ldots > r_k = 1$ of precisions. We then prove a lower bound of the complexity $K^A_{r_{i+1}, r_i}(\vert x - y\vert \mid \vert x - y\vert)$ on each interval of the resulting partition of $[1, r]$. 

\medskip

\noindent \textbf{Constructing the partition $\mathcal{P}$}: We define the sequence $\{r_i\}$ inductively. To begin, we set $r_0 = r$. Having defined the sequence up to $r_i$, we choose $r_{i+1}$ as follows. Let $a \leq r_i$ be the minimal real such that $[a, r_i]$ can be written as the union of yellow intervals whose lengths are at most doubling. If $a < r_i$, then we set $r_{i+1} = a$. In this case, we will refer to $[r_{i+1},r_i]$ as a \textbf{yellow} interval of $\mathcal{P}$.

Otherwise, let $t_i < r_i$ be the maximum of all reals such that
\begin{equation}\label{eq:choiceOfRi+1}
    f(t) = f(r_i) + \frac{D - 1}{D}\left(dr_i - (d +1)t\right) - d(r_i-t).
\end{equation}
Let $t_i^\prime < r_i$ be the largest real such that $f(r_i) = f(t^\prime_i) + r_i - t_i^\prime$. Note that such a $t^\prime$ must exist. We then set 
\begin{equation}
    r_{i+1} = \max\{t_i, t_i^\prime\}.
\end{equation}
Note that, in this case, $[r_{i+1}, r_i]$ is teal. We therefore refer to intervals of this case as \textbf{teal} intervals of $\mathcal{P}$.

We begin by observing that our partition is well defined. 
\begin{obs}\label{obs:wellDefined}
Suppose that $r_i\leq r$ and $r_i > C$, for some fixed constant $C$ depending only on $y$. Then there is at least one $t$ such that
\begin{equation*}
    f(t) = f(r_i) + \frac{D - 1}{D}\left(dr_i - (d +1)t\right) - d(r_i-t).
\end{equation*}
\end{obs}
\begin{proof}
We first note that $f(0) = O(1)$, and so
    \begin{align*}
        f(r_i) + \frac{D - 1}{D}\left(dr_i - (d +1)t\right) - d(r_i-t) &= f(r_i) + d\frac{D - 1}{D}r_i - dr_i\\
        &= f(r_i) - \frac{d}{D}r_i\\
        &> f(0).
\end{align*}

We also see that, at $t = r_i$,
    \begin{equation*}
        f(r_i) + \frac{D - 1}{D}\left(dr_i - (d +1)t\right) - d(r_i-t) < f(r_i),
    \end{equation*}
and so by the mean value theorem, the conclusion follows.
\end{proof}

We now show that a partition $\mathcal{P}$ does not contain too many intervals. This will allow us to control the accumulation of error terms when applying the symmetry of information.
\begin{lem}\label{lem:notTooManyIntervals}
If $[r_{i+1}, r_i] \in \mathcal{P}$ is teal, then $r_{i+1} \leq \frac{r_i}{2}$.
\end{lem}
\begin{proof}
Suppose that $[r_{i+1}, r_i] \in \mathcal{P}$ is teal. Then, by the construction of $\mathcal{P}$, $[t, r_i]$ is not yellow, for any $\frac{r_i}{2}\leq t < r_i$. This immediately implies that $t^\prime_i < \frac{r_i}{2}$. Moreover, for any $t > \frac{r_{i}}{2}$, we see that
\begin{align*}
    f(t) - f(r_i) + \frac{d}{D}r_i - \frac{d+1-D}{D}t &\geq \frac{d+1}{D}t - \frac{D-d}{D}r_i\\
    > 0,
\end{align*}
implying that $t_i \leq \frac{r_i}{2}$, and the conclusion follows.
\end{proof}

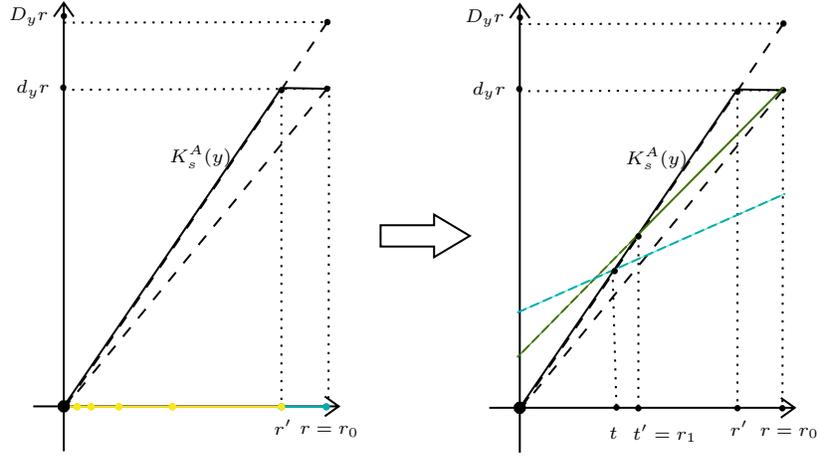
\begin{figure}[t]

\centering

\tikzset{every picture/.style={line width=0.75pt}} 

\begin{tikzpicture}[x=0.75pt,y=0.75pt,yscale=-1,xscale=1]

\draw  (204.81,269.3) -- (358.89,269.3)(220.22,65.62) -- (220.22,291.93) (351.89,264.3) -- (358.89,269.3) -- (351.89,274.3) (215.22,72.62) -- (220.22,65.62) -- (225.22,72.62)  ;
\draw  [fill={rgb, 255:red, 0; green, 0; blue, 0 }  ,fill opacity=1 ] (351.92,109) .. controls (351.92,108.39) and (352.41,107.89) .. (353,107.89) .. controls (353.59,107.89) and (354.08,108.39) .. (354.08,109) .. controls (354.08,109.61) and (353.59,110.11) .. (353,110.11) .. controls (352.41,110.11) and (351.92,109.61) .. (351.92,109) -- cycle ;
\draw  [fill={rgb, 255:red, 0; green, 0; blue, 0 }  ,fill opacity=1 ] (352.1,75.35) .. controls (352.1,74.74) and (352.59,74.24) .. (353.18,74.24) .. controls (353.78,74.24) and (354.26,74.74) .. (354.26,75.35) .. controls (354.26,75.97) and (353.78,76.46) .. (353.18,76.46) .. controls (352.59,76.46) and (352.1,75.97) .. (352.1,75.35) -- cycle ;
\draw  [dash pattern={on 4.5pt off 4.5pt}]  (353.18,75.35) -- (220.22,270.41) ;
\draw  [dash pattern={on 4.5pt off 4.5pt}]  (353,109) -- (220.22,270.41) ;
\draw  [fill={rgb, 255:red, 0; green, 0; blue, 0 }  ,fill opacity=1 ] (351.57,269.56) .. controls (351.56,268.94) and (352.04,268.44) .. (352.63,268.43) .. controls (353.23,268.42) and (353.72,268.9) .. (353.73,269.52) .. controls (353.74,270.13) and (353.26,270.64) .. (352.67,270.65) .. controls (352.07,270.66) and (351.58,270.17) .. (351.57,269.56) -- cycle ;
\draw  [fill={rgb, 255:red, 0; green, 0; blue, 0 }  ,fill opacity=1 ] (218.9,108.41) .. controls (218.89,107.8) and (219.37,107.29) .. (219.96,107.28) .. controls (220.56,107.27) and (221.05,107.76) .. (221.06,108.37) .. controls (221.07,108.98) and (220.59,109.49) .. (220,109.5) .. controls (219.41,109.51) and (218.91,109.02) .. (218.9,108.41) -- cycle ;
\draw  [fill={rgb, 255:red, 0; green, 0; blue, 0 }  ,fill opacity=1 ] (219.19,72.32) .. controls (219.18,71.7) and (219.66,71.2) .. (220.25,71.19) .. controls (220.85,71.18) and (221.34,71.67) .. (221.35,72.28) .. controls (221.36,72.89) and (220.89,73.4) .. (220.29,73.41) .. controls (219.7,73.42) and (219.21,72.93) .. (219.19,72.32) -- cycle ;
\draw  [dash pattern={on 0.84pt off 2.51pt}]  (354.06,107.87) -- (353.73,269.52) ;
\draw  [dash pattern={on 0.84pt off 2.51pt}]  (220,109.5) -- (353,109) ;
\draw  [dash pattern={on 0.84pt off 2.51pt}]  (220.27,75.63) -- (353.18,75.35) ;
\draw  [dash pattern={on 0.84pt off 2.51pt}]  (330,108.6) -- (330,269.5) ;
\draw  [fill={rgb, 255:red, 0; green, 0; blue, 0 }  ,fill opacity=1 ] (328.92,109.71) .. controls (328.92,109.1) and (329.41,108.6) .. (330,108.6) .. controls (330.59,108.6) and (331.08,109.1) .. (331.08,109.71) .. controls (331.08,110.33) and (330.59,110.82) .. (330,110.82) .. controls (329.41,110.82) and (328.92,110.33) .. (328.92,109.71) -- cycle ;
\draw [color={rgb, 255:red, 0; green, 173; blue, 173 }  ,draw opacity=1 ][fill={rgb, 255:red, 80; green, 227; blue, 194 }  ,fill opacity=1 ]   (330,269.5) -- (352.65,269.54) ;
\draw [shift={(352.65,269.54)}, rotate = 0.09] [color={rgb, 255:red, 0; green, 173; blue, 173 }  ,draw opacity=1 ][fill={rgb, 255:red, 0; green, 173; blue, 173 }  ,fill opacity=1 ][line width=0.75]      (0, 0) circle [x radius= 1.34, y radius= 1.34]   ;
\draw [shift={(330,269.5)}, rotate = 0.09] [color={rgb, 255:red, 0; green, 173; blue, 173 }  ,draw opacity=1 ][fill={rgb, 255:red, 0; green, 173; blue, 173 }  ,fill opacity=1 ][line width=0.75]      (0, 0) circle [x radius= 1.34, y radius= 1.34]   ;
\draw [color={rgb, 255:red, 248; green, 231; blue, 28 }  ,draw opacity=1 ][fill={rgb, 255:red, 248; green, 231; blue, 28 }  ,fill opacity=1 ]   (275,269.5) -- (330,269.5) ;
\draw [shift={(330,269.5)}, rotate = 0] [color={rgb, 255:red, 248; green, 231; blue, 28 }  ,draw opacity=1 ][fill={rgb, 255:red, 248; green, 231; blue, 28 }  ,fill opacity=1 ][line width=0.75]      (0, 0) circle [x radius= 1.34, y radius= 1.34]   ;
\draw [shift={(275,269.5)}, rotate = 0] [color={rgb, 255:red, 248; green, 231; blue, 28 }  ,draw opacity=1 ][fill={rgb, 255:red, 248; green, 231; blue, 28 }  ,fill opacity=1 ][line width=0.75]      (0, 0) circle [x radius= 1.34, y radius= 1.34]   ;
\draw [color={rgb, 255:red, 248; green, 231; blue, 28 }  ,draw opacity=1 ][fill={rgb, 255:red, 248; green, 231; blue, 28 }  ,fill opacity=1 ]   (248,269.5) -- (275,269.5) ;
\draw [shift={(275,269.5)}, rotate = 0] [color={rgb, 255:red, 248; green, 231; blue, 28 }  ,draw opacity=1 ][fill={rgb, 255:red, 248; green, 231; blue, 28 }  ,fill opacity=1 ][line width=0.75]      (0, 0) circle [x radius= 1.34, y radius= 1.34]   ;
\draw [shift={(248,269.5)}, rotate = 0] [color={rgb, 255:red, 248; green, 231; blue, 28 }  ,draw opacity=1 ][fill={rgb, 255:red, 248; green, 231; blue, 28 }  ,fill opacity=1 ][line width=0.75]      (0, 0) circle [x radius= 1.34, y radius= 1.34]   ;
\draw [color={rgb, 255:red, 248; green, 231; blue, 28 }  ,draw opacity=1 ][fill={rgb, 255:red, 248; green, 231; blue, 28 }  ,fill opacity=1 ]   (234,269.5) -- (248,269.5) ;
\draw [shift={(248,269.5)}, rotate = 0] [color={rgb, 255:red, 248; green, 231; blue, 28 }  ,draw opacity=1 ][fill={rgb, 255:red, 248; green, 231; blue, 28 }  ,fill opacity=1 ][line width=0.75]      (0, 0) circle [x radius= 1.34, y radius= 1.34]   ;
\draw [shift={(234,269.5)}, rotate = 0] [color={rgb, 255:red, 248; green, 231; blue, 28 }  ,draw opacity=1 ][fill={rgb, 255:red, 248; green, 231; blue, 28 }  ,fill opacity=1 ][line width=0.75]      (0, 0) circle [x radius= 1.34, y radius= 1.34]   ;
\draw [color={rgb, 255:red, 248; green, 231; blue, 28 }  ,draw opacity=1 ][fill={rgb, 255:red, 248; green, 231; blue, 28 }  ,fill opacity=1 ]   (227,269.5) -- (234,269.5) ;
\draw [shift={(234,269.5)}, rotate = 0] [color={rgb, 255:red, 248; green, 231; blue, 28 }  ,draw opacity=1 ][fill={rgb, 255:red, 248; green, 231; blue, 28 }  ,fill opacity=1 ][line width=0.75]      (0, 0) circle [x radius= 1.34, y radius= 1.34]   ;
\draw [shift={(227,269.5)}, rotate = 0] [color={rgb, 255:red, 248; green, 231; blue, 28 }  ,draw opacity=1 ][fill={rgb, 255:red, 248; green, 231; blue, 28 }  ,fill opacity=1 ][line width=0.75]      (0, 0) circle [x radius= 1.34, y radius= 1.34]   ;
\draw [color={rgb, 255:red, 248; green, 231; blue, 28 }  ,draw opacity=1 ][fill={rgb, 255:red, 248; green, 231; blue, 28 }  ,fill opacity=1 ]   (220.22,269.3) -- (227,269.5) ;
\draw [shift={(227,269.5)}, rotate = 1.71] [color={rgb, 255:red, 248; green, 231; blue, 28 }  ,draw opacity=1 ][fill={rgb, 255:red, 248; green, 231; blue, 28 }  ,fill opacity=1 ][line width=0.75]      (0, 0) circle [x radius= 1.34, y radius= 1.34]   ;
\draw [shift={(220.22,269.3)}, rotate = 1.71] [color={rgb, 255:red, 248; green, 231; blue, 28 }  ,draw opacity=1 ][fill={rgb, 255:red, 248; green, 231; blue, 28 }  ,fill opacity=1 ][line width=0.75]      (0, 0) circle [x radius= 1.34, y radius= 1.34]   ;
\draw    (330,108.6) -- (220.22,269.3) ;
\draw  [fill={rgb, 255:red, 0; green, 0; blue, 0 }  ,fill opacity=1 ][line width=3]  (219.14,269.31) .. controls (219.13,268.7) and (219.61,268.19) .. (220.2,268.19) .. controls (220.8,268.18) and (221.29,268.67) .. (221.29,269.28) .. controls (221.3,269.9) and (220.82,270.4) .. (220.23,270.41) .. controls (219.64,270.41) and (219.15,269.92) .. (219.14,269.31) -- cycle ;
\draw    (330,108.6) -- (353,109) ;
\draw  (434.81,270.02) -- (588.89,270.02)(450.22,66.35) -- (450.22,292.65) (581.89,265.02) -- (588.89,270.02) -- (581.89,275.02) (445.22,73.35) -- (450.22,66.35) -- (455.22,73.35)  ;
\draw  [fill={rgb, 255:red, 0; green, 0; blue, 0 }  ,fill opacity=1 ] (581.92,109.72) .. controls (581.92,109.11) and (582.41,108.61) .. (583,108.61) .. controls (583.59,108.61) and (584.08,109.11) .. (584.08,109.72) .. controls (584.08,110.34) and (583.59,110.83) .. (583,110.83) .. controls (582.41,110.83) and (581.92,110.34) .. (581.92,109.72) -- cycle ;
\draw  [fill={rgb, 255:red, 0; green, 0; blue, 0 }  ,fill opacity=1 ] (582.1,76.08) .. controls (582.1,75.46) and (582.59,74.97) .. (583.18,74.97) .. controls (583.78,74.97) and (584.26,75.46) .. (584.26,76.08) .. controls (584.26,76.69) and (583.78,77.19) .. (583.18,77.19) .. controls (582.59,77.19) and (582.1,76.69) .. (582.1,76.08) -- cycle ;
\draw  [dash pattern={on 4.5pt off 4.5pt}]  (583.18,76.08) -- (450.22,271.13) ;
\draw  [dash pattern={on 4.5pt off 4.5pt}]  (583,109.72) -- (450.22,271.13) ;
\draw  [fill={rgb, 255:red, 0; green, 0; blue, 0 }  ,fill opacity=1 ] (581.57,270.28) .. controls (581.56,269.67) and (582.04,269.16) .. (582.63,269.15) .. controls (583.23,269.14) and (583.72,269.63) .. (583.73,270.24) .. controls (583.74,270.85) and (583.26,271.36) .. (582.67,271.37) .. controls (582.07,271.38) and (581.58,270.89) .. (581.57,270.28) -- cycle ;
\draw  [fill={rgb, 255:red, 0; green, 0; blue, 0 }  ,fill opacity=1 ] (448.9,109.13) .. controls (448.89,108.52) and (449.37,108.02) .. (449.96,108) .. controls (450.56,107.99) and (451.05,108.48) .. (451.06,109.09) .. controls (451.07,109.71) and (450.59,110.21) .. (450,110.22) .. controls (449.41,110.23) and (448.91,109.75) .. (448.9,109.13) -- cycle ;
\draw  [fill={rgb, 255:red, 0; green, 0; blue, 0 }  ,fill opacity=1 ] (449.19,73.04) .. controls (449.18,72.43) and (449.66,71.92) .. (450.25,71.91) .. controls (450.85,71.9) and (451.34,72.39) .. (451.35,73) .. controls (451.36,73.62) and (450.89,74.12) .. (450.29,74.13) .. controls (449.7,74.14) and (449.21,73.65) .. (449.19,73.04) -- cycle ;
\draw  [dash pattern={on 0.84pt off 2.51pt}]  (583,109.72) -- (582.67,271.37) ;
\draw  [dash pattern={on 0.84pt off 2.51pt}]  (450,110.22) -- (583,109.72) ;
\draw  [dash pattern={on 0.84pt off 2.51pt}]  (450.27,76.35) -- (583.18,76.08) ;
\draw  [dash pattern={on 0.84pt off 2.51pt}]  (560,109.33) -- (560,270.22) ;
\draw  [fill={rgb, 255:red, 0; green, 0; blue, 0 }  ,fill opacity=1 ] (558.92,110.44) .. controls (558.92,109.83) and (559.41,109.33) .. (560,109.33) .. controls (560.59,109.33) and (561.08,109.83) .. (561.08,110.44) .. controls (561.08,111.05) and (560.59,111.55) .. (560,111.55) .. controls (559.41,111.55) and (558.92,111.05) .. (558.92,110.44) -- cycle ;
\draw    (560,109.33) -- (450.22,270.02) ;
\draw  [fill={rgb, 255:red, 0; green, 0; blue, 0 }  ,fill opacity=1 ][line width=3]  (449.14,270.03) .. controls (449.13,269.42) and (449.61,268.92) .. (450.2,268.91) .. controls (450.8,268.9) and (451.29,269.4) .. (451.29,270.01) .. controls (451.3,270.62) and (450.82,271.12) .. (450.23,271.13) .. controls (449.64,271.14) and (449.15,270.65) .. (449.14,270.03) -- cycle ;
\draw    (560,109.33) -- (583,109.72) ;
\draw [color={rgb, 255:red, 65; green, 117; blue, 5 }  ,draw opacity=1 ][fill={rgb, 255:red, 65; green, 117; blue, 5 }  ,fill opacity=1 ]   (583,108.61) -- (509,183.5) ;
\draw [color={rgb, 255:red, 65; green, 117; blue, 5 }  ,draw opacity=1 ][fill={rgb, 255:red, 65; green, 117; blue, 5 }  ,fill opacity=1 ] [dash pattern={on 3.75pt off 1.5pt}]  (509,183.5) -- (449,244.5) ;
\draw  [fill={rgb, 255:red, 0; green, 0; blue, 0 }  ,fill opacity=1 ] (509,183.5) .. controls (509,182.89) and (509.48,182.39) .. (510.08,182.39) .. controls (510.67,182.39) and (511.15,182.89) .. (511.15,183.5) .. controls (511.15,184.11) and (510.67,184.61) .. (510.08,184.61) .. controls (509.48,184.61) and (509,184.11) .. (509,183.5) -- cycle ;
\draw  [dash pattern={on 0.84pt off 2.51pt}]  (510.08,183.5) -- (510,271.39) ;
\draw [color={rgb, 255:red, 0; green, 173; blue, 173 }  ,draw opacity=1 ][fill={rgb, 255:red, 80; green, 227; blue, 194 }  ,fill opacity=1 ] [dash pattern={on 3.75pt off 1.5pt}]  (449,222) -- (584,162) ;
\draw  [dash pattern={on 0.84pt off 2.51pt}]  (497,201) -- (499,271) ;
\draw  [fill={rgb, 255:red, 0; green, 0; blue, 0 }  ,fill opacity=1 ] (558.92,270.22) .. controls (558.92,269.61) and (559.41,269.11) .. (560,269.11) .. controls (560.59,269.11) and (561.08,269.61) .. (561.08,270.22) .. controls (561.08,270.84) and (560.59,271.33) .. (560,271.33) .. controls (559.41,271.33) and (558.92,270.84) .. (558.92,270.22) -- cycle ;
\draw  [fill={rgb, 255:red, 0; green, 0; blue, 0 }  ,fill opacity=1 ] (508.92,270.28) .. controls (508.92,269.67) and (509.41,269.17) .. (510,269.17) .. controls (510.59,269.17) and (511.08,269.67) .. (511.08,270.28) .. controls (511.08,270.89) and (510.59,271.39) .. (510,271.39) .. controls (509.41,271.39) and (508.92,270.89) .. (508.92,270.28) -- cycle ;
\draw  [fill={rgb, 255:red, 0; green, 0; blue, 0 }  ,fill opacity=1 ] (497.92,269.89) .. controls (497.92,269.28) and (498.41,268.78) .. (499,268.78) .. controls (499.59,268.78) and (500.08,269.28) .. (500.08,269.89) .. controls (500.08,270.5) and (499.59,271) .. (499,271) .. controls (498.41,271) and (497.92,270.5) .. (497.92,269.89) -- cycle ;
\draw  [fill={rgb, 255:red, 0; green, 0; blue, 0 }  ,fill opacity=1 ] (497,201) .. controls (497,200.39) and (497.48,199.89) .. (498.08,199.89) .. controls (498.67,199.89) and (499.15,200.39) .. (499.15,201) .. controls (499.15,201.61) and (498.67,202.11) .. (498.08,202.11) .. controls (497.48,202.11) and (497,201.61) .. (497,201) -- cycle ;
\draw   (380,177.88) -- (407,177.88) -- (407,172.5) -- (425,183.25) -- (407,194) -- (407,188.63) -- (380,188.63) -- cycle ;

\draw (272.18,136.29) node [anchor=north west][inner sep=0.75pt]  [font=\scriptsize]  {$K^A_{s}( y)$};
\draw (337.91,277.68) node [anchor=north west][inner sep=0.75pt]  [font=\scriptsize]  {$r=r_{0}$};
\draw (194.86,103.17) node [anchor=north west][inner sep=0.75pt]  [font=\scriptsize]  {$d_{y} r$};
\draw (191.17,66.54) node [anchor=north west][inner sep=0.75pt]  [font=\scriptsize]  {$D_{y} r$};
\draw (324.91,274.4) node [anchor=north west][inner sep=0.75pt]  [font=\scriptsize]  {$r'$};
\draw (502.18,137.02) node [anchor=north west][inner sep=0.75pt]  [font=\scriptsize]  {$K^A_{s}( y)$};
\draw (569.91,278.4) node [anchor=north west][inner sep=0.75pt]  [font=\scriptsize]  {$r=r_{0}$};
\draw (424.86,103.9) node [anchor=north west][inner sep=0.75pt]  [font=\scriptsize]  {$d_{y} r$};
\draw (421.17,67.27) node [anchor=north west][inner sep=0.75pt]  [font=\scriptsize]  {$D_{y} r$};
\draw (554.91,275.12) node [anchor=north west][inner sep=0.75pt]  [font=\scriptsize]  {$r'$};
\draw (505.92,276.68) node [anchor=north west][inner sep=0.75pt]  [font=\scriptsize]  {$t'=r_{1}$};
\draw (493.92,278.4) node [anchor=north west][inner sep=0.75pt]  [font=\scriptsize]  {$t$};

\end{tikzpicture}

\caption{On the left, the adversary complexity function with $d_y=1.2$ and $D_y=1.44$, equipped with an initial good partition of yellow intervals from $1$ to $r^\prime$ and a teal interval $[r^\prime, r]$. On the right, generating an improved partition from the good partition. Note that the green interval $[r_1, r_0]$ is more than doubling. By combining it with a good partition of $[1, r_1]$-- which will clearly be all-yellow-- we obtain an all yellow partition of $[1, r]$. $1.44>2\cdot 1.2-1$, so this improves over the previous subsection.}
\label{fig:betterallyellow}
\end{figure}

\begin{lem}\label{lem:leftEndpointNotTooSmall}
    Let $\ve>0$. Suppose that $r_{i+1}$ is sufficiently large and $[r_{i+1}, r_i]$ is a teal interval of the above partition. Then 
    \begin{equation*}
        r_{i+1} \geq \frac{d(D-1)}{D^2 + D - d - 1}r_i-\ve r_i.
    \end{equation*}
    In particular,
    \begin{equation*}
        r_{i+1} \geq \frac{d(2-D)}{2+d(2-D)}r_i+1.
    \end{equation*}
\end{lem}
\begin{proof}
Assume that $r_{i+1}$ is large enough that, for all $s>r_{i+1}$, 
\begin{equation*}
d s-\ve^\prime s \leq K^A_s(y) \leq D s+\ve^\prime s
\end{equation*}
\noindent for some $\ve^\prime$ to be determined. By our choice of $r_{i+1}$, 
\begin{equation}
        K^A_{r_{i+1}}(y) \geq K^A_{r_i}(y) - \frac{d}{D}r_i + \frac{1+d-D}{D}r_{i+1} - \ve^\prime r_i.
    \end{equation}
    We then have
    \begin{align*}
        D r_{i+1} &\geq D_y r_{i+1}\\
        &\geq K^A_{r_{i+1}}(y)-\ve^\prime r_{i+1}\\
        &\geq K^A_{r_i}(y) - \frac{d}{D}r_i + \frac{1+d-D-D\ve^\prime}{D}r_{i+1}\\
        &\geq dr_i-\ve^\prime r_i - \frac{d}{D}r_i + \frac{1+d-D-D\ve^\prime}{D}r_{i+1}\\
        &= \frac{d(D-1)-D\ve^\prime}{D}r_i + \frac{1+d-D-D\ve^\prime}{D}r_{i+1}.
    \end{align*}
    Rearranging and simplifying yields 
    \begin{equation}
        r_{i+1} \geq \frac{d(D-1)-D\ve^\prime}{D^2 + D - d - 1-D\ve^\prime}r_i.
    \end{equation}
\noindent Given $\ve>0$, choosing $\ve^\prime$ sufficiently small compared to these quantities gives the desired conclusion:

    \begin{equation}
        r_{i+1} \geq \frac{d(D-1)}{D^2 + D - d - 1}r_i-\ve r_i.
    \end{equation}

It is straightforward to verify that this implies that
\begin{equation*}
     r_{i+1} \geq \frac{d(2-D)}{2+d(2-D)}r_i
\end{equation*}
for sufficiently small $\ve$. Indeed, using the above inequality, it suffices to show that 
    \begin{equation*}
        \frac{d(D-1)}{D^2 + D - d - 1} > \frac{d(2-D)}{2+d(2-D)},
    \end{equation*}
or, equivalently, $d(2-D) > D -D^2+1$.

By our assumption, $d > 1$, and so the above inequality follows. Thus, for $r$ sufficiently large, we have
\begin{equation*}
    r_{i+1} \geq \frac{d(2-D)}{2+d(2-D)}r_i+1.
\end{equation*}
\end{proof}

\begin{lem}\label{lem:boundComplexityYellowNewPartition}
Let $\ve > 0$. Suppose that $[r_{i+1}, r_i]\in \mathcal{P}$ is a yellow interval. Then,
\begin{equation*}
    K^{A,x}_{r_i, r_{i+1}}(\vert x -y\vert \mid \vert x-y\vert) \geq r_i - r_{i+1} - \ve r_i.
\end{equation*}
\end{lem}
\begin{proof}
    By assumption, $[r_{i+1}, r_i]$ is the union of of yellow intervals $[a_{j+1}, a_j]$ such that $a_j \leq 2a_{j+1}$. By a simple greedy strategy we can construct a partition $P_1 = \{[b_{k+1}, b_k]\}$ of $[r_{i+1}, r_i]$ such that, for every $k$, $[b_{k+1}, b_k]$ is yellow, $b_k \leq 2 b_{k+1}$ and $b_{k+2} > 2b_k$. That is, $P_1$ is a good partition of $[r_{i+1}, r_i]$. The conclusion then follows from Lemma \ref{lem:boundGoodPartitionDistance}.
\end{proof}

\begin{lem}\label{lem:slopeBoundTeal}
    Let $\ve>0$ be given and suppose that $[r_{i+1}, r_i]$ is teal and $r_i$ is sufficiently large. Then
    \begin{equation}
        \frac{K^A_{r_i, r_{i+1}}(y\mid y)}{r_i - r_{i+1}} \geq \min\{1, \frac{d(2D - d -1)}{D^2+D-Dd-1}-\ve\}
    \end{equation}
\end{lem}
\begin{proof}
    Recall that we chose $r_{i+1}$ to be 
    \begin{center}
        $r_{i+1} = \max\{t, t^\prime\}$,
    \end{center}
    where $t^\prime$ is the largest real such that $[t^\prime, r_i]$ is green, and $t$ is the largest real such that
    \begin{equation*}
         f(t) = f(r_i) + \frac{D - 1}{D}\left(dr_i - (d +1)t\right) - d(r_i-t).
    \end{equation*}
    If $r_{i+1} = t^\prime$, then 
    \begin{equation*}
        K^A_{r_{i+1}}(y) = K^A_{r_i}(y) - (r_i - r_{i+1}),
    \end{equation*}
    and the conclusion holds trivially.

    We now assume that $r_{i+1}=t$. Then, by the previous proposition,
    \begin{equation}
        r_{i+1} \geq \frac{d(D-1)}{D^2 + D - d - 1}r_i-\ve r_i.
    \end{equation}
We proceed via a tedious, but straightforward, calculation. Noting that our interval is not green, we have from the definition of $r_{i+1}$ that 
\begin{equation*}
    K^A_{r_i}(y)-K^A_{r_{i+1}}(y) =- \frac{D - 1}{D}\left(dr_i - (d +1)r_{i+1}\right) + d(r_i-r_{i+1}).
\end{equation*}
Using this condition and the above bound on $r_{i+1}$ allows one to bound the growth rate on such an interval. We omit the details since the algebra becomes quite unpleasant.
\end{proof}

\begin{obs}\label{obs:generalpartitionfulldimension}
When $D\leq\frac{(3+\sqrt{5})d-1-\sqrt{5}}{2}$, $\frac{K^A_{r_i, r_{i+1}}(y\mid y)}{r_i - r_{i+1}} = 1$ whenever $[r_{i+1},r_i]$ is teal. 
\end{obs}

The last goal of this subsection is to prove that these particular teals are useful, namely, that we can lower bound the growth of the complexity of the distance by the growth of the complexity of $y$ on them. We start by stating the following lemma from \cite{Stull22c}:

\begin{lem}\label{lem:lowerBoundOtherPointDistance}
Let $x, y\in \R^2$ and $r\in \N$. Let $z\in \R^2$ such that $\vert x-y\vert = \vert x-z\vert$. Then for every $A\subseteq \N$,
\begin{equation}
K^A_r(z) \geq K^A_t(y) + K^A_{r-t, r}(x\mid y) - K_{r-t}(x\mid p_{e^\prime} x, e^\prime) - O(\log r),
\end{equation}
where $e^\prime = \frac{y-z}{\vert y-z\vert}$ and $t = -\log \vert y-z\vert$.
\end{lem}

Though it may appear somewhat cumbersome, the above lemma is a relatively straightforward consequence of attempting to compute $x$ given access to $y$ up to a certain precision through the use of $z$ and $w$ -- the midpoint between $y$ and $z$ -- which has the property that $p_{e^\prime}x=p_{e^\prime}w$, which is a key connection between projections and distances. In particular, note the term $K_{r-t}(x\mid p_{e^\prime} x, e^\prime)$ above; bounding this is where the projection theorem will be useful. Now, we state the final lemma of this subsection.

\begin{lem}\label{lem:goodtealgrowth}
Suppose that $[r_{i+1}, r_i] \in \mathcal{P}$ is a teal interval. For any $\ve>0$, provided that $r_{i+1}$ is sufficiently large, we have 
\begin{equation}
    K^{A,x}_{r_i, r_i, r_{i+1}}(y \mid \vert x - y \vert, y) \leq \ve r_i.
\end{equation}
Therefore,
$K^{A,x}_{r_i, r_{i+1}}(\vert x - y\vert \mid \vert x - y\vert ) \geq K^{A,x}_{r_i, r_{i+1}}(y\mid y) - \ve r_{i}$.
\end{lem}

Notice that the conclusion of Lemma \ref{lem:pointDistance} is almost exactly the conclusion of this lemma. Thus, we need to verify that its conditions are satisfied, which is the content of this proof. Essentially, this entails proving a lower bound on the complexity of points $z$ which are the same distance from $x$ as $y$ is. 
\begin{proof}
Let some small rational $\ve>0$ be given, and assume $r_{i+1}$ is sufficiently large. Let $\eta$ be the rational such that $\eta r_i=K^A_{r_i}(y)-4\ve$. Let $G=D(r, y, \eta)$ be the oracle of Lemma \ref{lem:oracles} relative to $A$. 

Our goal is to apply Lemma \ref{lem:pointDistance}. It is routine to verify that condition (i) of Lemma \ref{lem:pointDistance} holds. We must therefore verify condition (ii) That is, we need to show that, for any $z\in B_{2^{-r_{i+1}}}(y)$ whose distance from $x$ is $\vert x - y\vert$, either (i) $K^{A,G}_{r_i}(z)$ is greater than $\eta r_i$ or (ii) $z$ is very close to $y$. Formally, we must show that, for any such $z$,
\begin{equation}\label{eq:goodtealgrowth4}
    K_{r_i}^{A, G}(z)\geq \eta r_i + \min\{\ve r_i, r_i-s-\ve r_i\}, 
\end{equation}
where $s=-\log \vert y-z\vert$.

To that end, let $z\in B_{2^{-r_{i+1}}}(y)$ such that $\vert x-y\vert=\vert x-z\vert$. Let $s=\vert y-z\vert$. We consider two cases. For the first, assume that $s\geq \frac{r_i}{2}-\log r_i$. Then, as observed in \cite{Stull22c}, the projections of $y$ and $z$ in the direction $e$ are almost exactly the same. Specifically, $\vert p_ey - p_e z\vert< r_i^2 2^{-r_i}$. Then, letting $r_i^\prime=r_i-2 \log r_i$, these projections are indistinguishable at precision $r_i^\prime$. This enables us to apply Lemma \ref{lem:lowerBoundOtherPoints} which, in conjunction with property (C4) and the properties of our oracle $G$ imply that
\begin{equation}
    K_{r_i}^{A, G}(z)\geq K_s^{A, G}(y) + r_i - s - \frac{\ve}{2}r_i-O(\log r_i)
\end{equation}

\noindent Then, using the fact that $K_s^{A, G}(y)=\min\{\eta r_i, K_s^{A}(y) + O(\log r_i)$ and considering each of these cases establishes (\ref{eq:goodtealgrowth4}) in the case that $s \geq \frac{r_i}{2}-\log r_i$.

This leaves the case that $s<\frac{r_i}{2}-\log r_i$. Note that this immediately implies that $K^{A,G}_s(y) = K^A_s(y) - O(\log r_i)$. Lemma \ref{lem:lowerBoundOtherPointDistance}, relative to $(A, G)$, implies that 
\begin{equation}
K^{A, G}_{r_i}(z) \geq K^{A, G}_s(y) + K^{A, G}_{r_i-s, r_i}(x\mid y) - K^{A, G}_{r_i-s}(x\mid p_{e^\prime} x, e^\prime) - O(\log r). 
\end{equation}
\noindent To bound the projection term, we need to apply Theorem \ref{thm:modifiedProjectionTheorem} with respect to $x$, $e^\prime$, $\ve$, a constant $C$ (depending only on $x$ and $y$), $t=s$, and $r=r_i-s$. We now check that the conditions are satisfied. 

First, observe that $r_{i+1}-1<s<\frac{r_i}{2} - \log r_i$, since $z$ is assumed to be within $2^{-r_{i+1}}$ of $y$. The second inequality implies that we can take $r_i-s$ to be sufficiently large, since $r_i$ is taken to be sufficiently large. From the first inequality and Lemma \ref{lem:leftEndpointNotTooSmall}, we obtain that 
\begin{equation}
s\geq\left(\frac{d(2-D)}{2+d(2-D)}r_i+1\right)-1.
\end{equation}
Hence,
\begin{equation}
s\geq \frac{d(2-D)}{2}(r_i - s)
\end{equation}
 Thus conditions (P1) and (P2) are satisfied. As for condition (P3), from an observation in \cite{Stull22c}, $e^\prime$ and $e$ are close enough to each other that, using the fact that $e$ and its orthogonal complement are computable from each other, we have for $s^\prime\leq s$
 \begin{equation}
K^{A, x}_{s^\prime}(e^\prime)=K^{A, x}_{s^\prime}(e)+O(\log s^\prime). 
 \end{equation}
So, using condition (C2), we have
 \begin{equation}
K^{A, x}_{s^\prime}(e^\prime)\geq s^\prime - C \log s^\prime 
 \end{equation}
and thus we may apply Theorem \ref{thm:modifiedProjectionTheorem}.

Using Theorem \ref{thm:modifiedProjectionTheorem} the properties of $G$, and our choice of $r_{i+1}$ yields
\begin{align*}
K^{A, G}_{r_i}(z) &\geq K^{A, G}_s(y) + K^{A, G}_{r_i-s, r_i}(x\mid y) - K^{A, G}_{r_i-s}(x\mid p_{e^\prime} x, e^\prime) - O(\log r)\\
&\geq K^{A}_s(y) + K^{A}_{r_i-s, r_i}(x\mid y) - K^{A, G}_{r_i-s}(x\mid p_{e^\prime} x, e^\prime) - O(\log r)\\
&\geq K^{A}_s(y) + K^{A}_{r_i-s}(x) - K^{A, G}_{r_i-s}(x\mid p_{e^\prime} x, e^\prime) - O(\log r)\\
&\geq K^{A}_s(y) + K^{A}_{r_i-s}(x) - K^A_{r_i-s}(x\mid p_{e^\prime} x, e^\prime) - O(\log r)\\
&\geq K^{A}_s(y) + K^{A}_{r_i-s}(x)  - \ve r_i-O(\log r)\\ 
&- \max\{K^A_{r_i-s}(x) - \frac{d}{D}r_i+\frac{d+1-D}{D}s, K^A_{r_i-s}(x) - (r_i-s)\} 
\end{align*}
Hence, 
\begin{equation}\label{eq:goodTealGrowth1}
K^{A, G}_{r_i}(z) \geq K^{A}_s(y)  
+ \min\{\frac{d}{D}r_i-\frac{d+1-D}{D}s,(r_i-s)\}- \ve r_i-O(\log r) 
\end{equation}

By our choice of $r_{i+1}$, (\ref{eq:choiceOfRi+1}), we see that
\begin{equation}\label{eq:goodTealGrowth2}
    K^A_s(y) \geq K^A_{r_i}(y) -\frac{d}{D}r_i + \frac{d+1 - D}{D}s.
\end{equation}
Combining (\ref{eq:goodTealGrowth1}) and (\ref{eq:goodTealGrowth2}) shows that
\begin{align*}
    K^{A, G}_{r_i}(z) &\geq K^{A}_s(y) + \min\{\frac{d}{D}r_i-\frac{d+1-D}{D}s,(r_i-s)\}- \ve r_i-O(\log r) \\
    &\geq  K^A_{r_i}(y) -\frac{d}{D}r_i + \frac{d+1 - D}{D}s \\
    &\;\;\;\;\;\;\;\;\;\;\;\;+ \min\{\frac{d}{D}r_i-\frac{d+1-D}{D}s,(r_i-s)\}- \ve r_i-O(\log r)
\end{align*}
If $\frac{d}{D}r_i-\frac{d+1-D}{D}s \leq r_i - s$, then we have
\begin{align*}
    K^{A, G}_{r_i}(z) &\geq K^A_{r_i}(y) - 2\ve r_i\\
    &\geq \eta r_i + \ve r_i,
\end{align*}
and (\ref{eq:goodtealgrowth4}) holds. Otherwise, since $[r_{i+1},r_i]$ is teal, 
\begin{align*}
     K^{A, G}_{r_i}(z) &\geq  K^A_s(y) + r_i-s- \ve r_i-O(\log r) \\
     &\geq K^A_{r_i}(y) - \ve r_i-O(\log r)\\
     &\geq \eta r_i + \ve r_i
\end{align*}
and we can again establish (\ref{eq:goodtealgrowth4}).

Therefore, we are able to apply Lemma \ref{lem:pointDistance}, which shows that 
\begin{align*}
      K^{A, x}_{r_i, r_{i+1}}(\vert x - y\vert \mid \vert x - y\vert) &\geq K^{A, x}_{r_i, r_{i+1}}(\vert x - y\vert \mid y)\\
      &\geq K^{A, G, x}_{r_i, r_{i+1}}(\vert x - y\vert \mid y)\\
     &\geq K^{A, G, x}_{r_i, r_{i+1}}(y\mid y) - 3\ve r + K(\ve, \eta) - O(\log r_i)\\
     &\geq K^{A, x}_{r_i, r_{i+1}}(y\mid y) - 4\ve r,
\end{align*}
and the proof is complete.
\end{proof}

\bigskip

\subsection{Main theorem for effective Hausdorff dimension}
In this section, we prove the point-wise analog of the main theorem of this paper. That is, we prove the following. 
\begin{thm}\label{thm:mainThmEffDim}
Suppose that $x, y\in\R^2$, $e = \frac{y-x}{\vert y-x\vert}$, and $A, B\subseteq\N$  satisfy the following.
\begin{itemize}
\item[\textup{(C1)}] $d_x, d_y > 1$
\item[\textup{(C2)}] $K^{x,A}_r(e) = r - O(\log r)$ for all $r$.
\item[\textup{(C3)}] $K^{x,A, B}_r(y) \geq K^{A}_r(y) - O(\log r)$ for all sufficiently large $r$. 
\item[\textup{(C4)}] $K^{A}_r(e\mid y) = r - o(r)$ for all $r$.
\end{itemize}
Then 
\begin{equation*}
\dim^{x,A}(\vert x-y\vert) \geq d\left(1 - \frac{(D-1)(D-d)}{2(D^2+D -1)-2d(2D-1)}\right),
\end{equation*}
where $d = \min\{d_x, d_y\}$ and $D = \max\{D_x, D_y\}$. Furthermore, if 
\begin{equation*}
D\leq\frac{(3+\sqrt{5})d-1-\sqrt{5}}{2}
\end{equation*}
Then $\dim^{x,A}(\vert x-y\vert)=1$.
\end{thm}

In the last subsection we have a good bound on the complexity growth of the distance on any teal interval, and the complexity growth on any yellow is 1. Then, it would seem that the worst case scenario is that our partition is (almost) all teal. But, this case is advantageous too, because if there is very little yellow, then almost all the complexity growth for $y$ has to take place on the teals, and Lemma \ref{lem:goodtealgrowth} indicates we can transfer \emph{all} the growth of $K_s^A(y)$ on teals to $K_s^A(\vert x - y\vert)$. So, the worst case scenario is actually when there is an intermediate amount of yellow. Now, we formalize this and prove the theorem. 

\begin{proof}

Let $\ve>0$ be given and let $r$ be sufficiently large. Let $\mathcal{P} = \{[r_{i+1}, r_i]\}$ be the partition of $[1, r]$ defined in the previous section. Let $L$ be the total length of the yellow intervals. Recall that if $[r_{i+1}, r_i]$ is yellow, then we have that
    \begin{equation*}
        K^{A,x}_{r_i, r_{i+1}}(\vert x - y\vert \mid \vert x - y\vert ) \geq r_i - r_{i+1} - \ve r_{i}
    \end{equation*}

By the previous lemma, and repeated applications of the symmetry of information, we have for sufficiently large $r$ that
\begin{align*}
    K^{A,x}_r(\vert x - y\vert ) &= \sum\limits_{i\in Y} K^{A,x}_{r_i, r_{i+1}}(\vert x - y\vert \mid \vert x - y\vert ) + \sum\limits_{i\in Y^C} K^{A,x}_{r_i, r_{i+1}}(\vert x - y\vert \mid \vert x - y\vert )\\
    &\geq L -\frac{\ve}{3} r+ \sum\limits_{i\in Y^C} K^{A,x}_{r_i, r_{i+1}}(\vert x - y\vert \mid \vert x - y\vert )\\
    &\geq L -\frac{2\ve}{3}r + \sum\limits_{i\in Y^C} K^{A,x}_{r_i, r_{i+1}}(y\mid y)\\
    &\geq L + \min\{1, \frac{d(2D-d-1)}{D^2 + D - Dd - 1}\}(r - L)-\ve r.\label{eq:notmuchteal}
\end{align*}

\noindent From Observation \ref{obs:generalpartitionfulldimension}, we have 1 as the minimum when $D\leq\frac{(3+\sqrt{5})d-1-\sqrt{5}}{2}$, so in this case $K^{A,x}_r(\vert x - y\vert )\geq L+(r-L)-\ve r=r -\ve r$. Taking $\ve$ as small as desired and letting $r$ go to infinity completes the proof. Now assume we are not in the above case and that $r$ is sufficiently large given $\ve$, and note the following bound which is advantageous when there is not much yellow:
\begin{align*}
    d_y r &\leq K^{A,x}_r(y)+\frac{\ve}{3}\\
    &= \sum\limits_{i\in Y} K^{A,x}_{r_i, r_{i+1}}(y\mid y) + \sum\limits_{i\in Y^C} K^{A,x}_{r_i, r_{i+1}}(y\mid y )\\
    &\leq 2L + \frac{2\ve}{3} +\sum\limits_{i\in Y^C} K^{A,x}_{r_i, r_{i+1}}(\vert x - y\vert )\\
    &= L + K^{A,x}_r(\vert x - y\vert ) +\ve r.
\end{align*}
Hence,
\begin{equation}\label{eq:distanceBoundTermsofL}
    K^{A,x}_r(\vert x - y\vert) \geq \max\{L + \frac{d(2D-d-1)}{D^2 + D - Dd - 1}(r - L), d r - L\}-\ve r.
\end{equation}

The first term is increasing in $L$ (since we are considering the case where $\frac{d(2D-d-1)}{D^2 + D - Dd - 1}<1$), and the second term is decreasing in $L$, so we can set them equal to find the minimum over all $L$, which yields

\begin{equation}
K^{x,A}_r(\vert x-y\vert) \geq d\left(1 - \frac{\left(D-1\right)\left(D-d\right)}{2D^2+\left(2-4d\right)D+d^2+d-2}\right) -\ve r.
\end{equation}

\noindent Since we can take $\ve$ as small as desired and then let $r$ go to infinity, this completes the proof.

\end{proof}

Now, we can consider a few special cases. First, note that when combined with Proposition \ref{prop:yellowfulldimension}, for any choice of $d$, our lower bound on the dimension is monotone decreasing in $D$. Thus, setting $D=2$ gives the following corollary:

\begin{cor}\label{cor:hausdorffbound1}
When conditions (C1)-(C4) are satisfied,
\begin{equation}
\dim^{x, A} (\vert x-y\vert) \geq \frac{d(d-4)}{d-5}.
\end{equation}
\end{cor}

Similarly, our lower bound is monotone increasing in $d$, so setting $d=1$ gives

\begin{cor}\label{cor:hausdorffbound2}
When conditions (C1)-(C4) are satisfied,

\begin{equation}
\dim^{x, A} (\vert x-y\vert) \geq \frac{D+1}{2D}.
\end{equation}
\end{cor}

 Note that these are the effective analogs of Corollary \ref{cor:firstMainCor} and Corollary \ref{cor:secondMainCor} in the introduction. The first of these corollaries is a helpful comparison point to previous work on the pinned distance problem, and the second will be useful in the next subsection.

\bigskip

\subsection{Main theorem for effective packing dimension}

We can use our work in the previous section to prove a new bound on the packing dimension of pinned distance sets. The basic idea is that for effective Hausdorff dimension, we had to prove a lower bound for $K_r^{x, A}(\vert x-y\vert)$ at \emph{every} sufficiently large precision (since dim is defined with a limit inferior), whereas for effective packing dimension, we are free to \emph{choose} a sequence of advantageous precisions (since Dim is defined with a limit superior). The idea is to consider ``maximal'' precisions $r_i$ where $K^A_{r_i}(y)\approx Dr_i$. These maximal precisions have to be contained in large yellow intervals, and prior to the large yellow intervals, we can use the bound from Corollary \ref{cor:hausdorffbound2}, which holds at \emph{every} precision. 

We call an interval $[a, b]$ an \textbf{all yellow} interval if there is a good partition of $[a,b]$ consisting entirely yellow intervals whose lengths are at most doubling.
\begin{lem}
   Suppose that $\ve > 0$, $A,B\subseteq\N$ and $x, y\in \R^2$ satisfy (C1)-(C4). In addition, assume that $\Dim^A(x) \leq \Dim^A(y)$. Let $r$ be a sufficiently large precision which is maximal in the sense that $K^{A, x}_r(y)\geq D_y r -\ve r$. Then 
    \begin{equation}\label{eq:boundDistanceMaximalPrec}
    K^{A,B,x}_r(\vert x - y\vert) \geq \frac{D^2 - D + 2}{2D}r - \ve r.
\end{equation}
Moreover, if $[r_1, r_2]$ is an all yellow interval containing $r$, then 
\begin{equation}\label{eq:lowerBoundr2}
   K^{A,B,x}_{r_2}(\vert x - y \vert )\geq r_2 - \frac{3D - D^2 - 2}{2D}r - \ve r.
\end{equation}
\end{lem}
\begin{proof}
 Let $r$ be sufficiently large such that
\begin{equation}
    K^{A}_r(y) \geq D r - \frac{\ve}{2} r.
\end{equation}
Let $\mathcal{P} = \{[r_{i+1}, r_i]\}$ be the partition of $[1,r]$ defined in the previous section. Let $L$ be the total length of the yellow intervals in $\mathcal{P}$. Using (\ref{eq:distanceBoundTermsofL}), we have that
\begin{equation}
    K^{A,x}_r(\vert x - y\vert) \geq \max\{L + \frac{2}{D+ 1}(r - L), D r - L\} -
    \ve r.
\end{equation}
We can therefore conclude that (\ref{eq:boundDistanceMaximalPrec}) holds.

Let $[r_1, r_2]$ be an all yellow interval containing $r$. Then, by Lemma \ref{lem:distancesYellowTeal},
\begin{align*}
    r_2 - r_1 - \frac{\ve r}{2} &\leq K^{A,x}_{r_2,r_1}(\vert x - y \vert \mid \vert x - y\vert)\\
    &\leq K^{A,x}_{r_2}(\vert x - y \vert ) - K^{A,x}_{r_1}(\vert x - y \vert) - O(\log r)\\
    &\leq K^{A,x}_{r_2}(\vert x - y \vert ) - \left(K^{A,x}_{r}(\vert x - y \vert) - (r-r_1)\right)- O(\log r).
\end{align*}
Rearranging, and using (\ref{eq:boundDistanceMaximalPrec}), we see that for sufficiently large $r$
\begin{align}
    K^{A,x}_{r_2}(\vert x - y \vert ) &= K^{A,x}_{r}(\vert x - y \vert) - (r-r_1) + r_2 - r_1\tag*{}-\frac{\ve}{2}r\\
    &\geq \frac{D^2 - D + 2}{2D}r - \ve r + r_2 - r\tag*{}\\
    &= r_2 - \frac{3D - D^2 - 2}{2D}r - \ve r\tag*{},
\end{align}
and the conclusion follows.
\end{proof}

\begin{thm}\label{thm:effectivePackingThm}
Suppose that $x, y\in\R^2$, $e = \frac{y-x}{\vert y-x\vert}$, and $A,B\subseteq\N$  satisfy $\Dim^A(x)\leq \Dim^A(y)$ and the following.
\begin{itemize}
\item[\textup{(C1)}] $d_x, d_y > 1$,
\item[\textup{(C2)}] $K^{x,A}_r(e) = r - O(\log r)$ for all $r$.
\item[\textup{(C3)}] $K^{x,A,B}_r(y) \geq K^{A}_r(y) - O(\log r)$ for all sufficiently large $r$. 
\item[\textup{(C4)}] $K^{A}_r(e\mid y) = r - o(r)$ for all $r$.
\end{itemize}

Then 
    $\Dim^{x, A, B}(\vert x - y\vert) \geq \frac{3D_y^2 - D_y + 6}{8D_y}$
\end{thm}
\begin{remark}
The extra requirement that $\Dim^A(x)\leq \Dim^A(y)$ seems necessary to obtain the largest lower bound with our methods. Previously in this section, we assumed $D$ was the larger of $D_x, D_y$, a safe assumption since a larger $D_x$ gives a worse projection theorem and a larger $D_y$ gives a worse adversary function when partitioning. For the packing bound, however, a higher $D_y$ could actually improve the bound. Thus, if $D_y$ were actually much smaller than $D$, this could make for a worse bound, which necessitates the extra assumption. 
\end{remark}

\begin{proof}
Let $\ve > 0$. Let $r$ be a sufficiently large precision which is maximal in the sense that $K^{A}_r(y)\geq D_y r -\ve r$. We also assume that the function $f$ associated to $K^A_s(y)$ is increasing on $[r-1, r]$. Note that there are infinitely many such $r$. We first assume that there exists an all yellow interval $[r_1, r_2]$ containing $r$ such that $r_2 \geq \frac{4}{3}r$. Then,
\begin{equation*}
    K^{A,x}_{r_2}(\vert x - y \vert ) \geq \frac{3D^2-D+6}{8D}r_2 - \ve r_2,
\end{equation*}
and the proof is complete. 

We now assume that, no such all yellow interval exists. Let $[r_1, r_2]$ be an all yellow interval containing $r$ which is of maximal length. This implies that there is an interval $[a, r_2]$ which is the union of green intervals, whose lengths are at most doubling, such that $a \leq \frac{r_2}{2}$. Hence $r \geq \frac{3}{2} a$. For convenience, we set $r^\prime = \frac{r_2}{2}$. Let $\mathcal{P} = \{[r_{i+1}, r_i]\}$ be the partition of $[1,r^\prime]$. Let $L$ be the total length of the yellow intervals in $\mathcal{P}$. We first see that
\begin{equation}
    K^{A,x}_{r^\prime}(\vert x - y\vert) \geq \max\{L + \frac{2}{D+ 1}(r^\prime - L), K^A_{r^\prime}(y) - L\}.
\end{equation}

Since $[a , r_2]$ is green, and $r \geq \frac{3}{2}r^\prime$,
\begin{align*}
    K^{A}_{r^\prime}(y) &\geq K^{A}_{2r^\prime}(y) - r^\prime\\
    &\geq K^{A}_r(y) -r^\prime\\
    &= Dr -r^\prime\\
    &> \frac{3D -2}{2}r^\prime.
\end{align*}
We now show that 
\begin{equation}\label{eq:lowerBoundAtRprime}
    K^{A,x}_{r^\prime}(\vert x - y\vert) \geq \frac{3D^2 - 5D+6}{4D}r^\prime - \ve r.
\end{equation}

If 
\begin{equation}
    K^{A}_{r^\prime}(y) - L \geq \frac{3D^2 - 5D+6}{4D}r^\prime - \ve r,
\end{equation}
then (\ref{eq:lowerBoundAtRprime}) holds immediately. Otherwise, using our lower bound on $K^{A}_{r^\prime}(y)$, we see that
\begin{equation}
    L \geq \frac{3D^2 + D - 6}{4D}r^\prime.
\end{equation}
Therefore,
\begin{align*}
    K^{A,x}_{r^\prime}(\vert x - y\vert) &\geq L + \frac{2}{D+1}(r^\prime - L)-\ve r\\
    &\geq \frac{3D^2 - 5D+6}{4D}r^\prime - \ve r,
\end{align*}
and so (\ref{eq:lowerBoundAtRprime}) holds. 

Since $[r_1, r_2]$ is all yellow, by Lemma \ref{lem:distancesYellowTeal} we see that
\begin{align*}
    K^{x,A, B}_{2r^\prime}(\vert x - y\vert) &\geq \frac{3D^2 - 5D+6}{4D}r^\prime + r^\prime-\ve r\\
    &= \frac{3D^2 - D+6}{4D}r^\prime -\ve r\\
    &\geq \frac{3D^2 - D+6}{8D}2r^\prime - 2 \ve r^\prime,
\end{align*}
Noting we can take $\ve$ as small as desired and then let $r$ go to infinity completes the proof. 
\end{proof}

\bigskip

\section{Dimensions of pinned distance sets}

In this section, we reduce the effective Hausdorff and packing theorems to their classical analogues. We prove Theorem \ref{thm:moregeneralmaintheorem} and then Theorem \ref{thm:regularYFullDim} in the first subsection, then consider packing dimension. Throughout, we have $x,y\in\R^2$, $e = \frac{x-y}{\vert x - y\vert}$ and $A,B \subseteq \N$. For ease of reference, we list conditions $(C1)-(C4)$ here:
\begin{itemize}
\item[\textup{(C1)}] $\dim^A(x)>d_x, d_y > 1$
\item[\textup{(C2)}] $K^{x,A}_r(e) = r - O(\log r)$ for all $r$.
\item[\textup{(C3)}] $K^{x,A,B}_r(y) \geq K^{A}_r(y) - O(\log r)$ for all sufficiently large $r$. 
\item[\textup{(C4)}] $K^{A}_r(e\mid y) = r - o(r)$ for all $r$.
\end{itemize}

Note the modification of condition (C1), here we drop the convention that $d_x=\dim^A(x)$. For the reduction, we will want $d_x$ to be strictly less than $\dim_H(X)$, because we will remove an exceptional set of size $d_x$. Thus, we want $\dim_A(x)>d_x$ so we can apply our effective theorem uniformly.

\bigskip

\subsection{Hausdorff dimension of pinned distance sets}
A main tool we need to reduce the classical statements to their effective counterparts is the following radial projection theorem of Orponen \cite{Orponen19DimSmooth}:
\begin{thm}\label{RadialProjection}
Let $Y\subseteq\mathbb{R}^2$ be a Borel set with $s=dim_H(Y)>1$ such that there is a measure $\mu\in \mathcal{M}(Y)$ such that there is a measure satisfying $I_d(\mu)<\infty$ for all $1<d<s$. Then there is a Borel $G\subseteq \mathbb{R}^2$ with $\dim_H(G)\leq 2-\dim_H(Y)$ such that, for every $x\in\mathbb{R}^2\setminus\text{spt}(\mu)$, $\mathcal{H}^1(\pi_x (Y))>0$. Moreover, the pushforward of $\mu$ under $\pi_x$ is absolutely continuous with respect to $\mathcal{H}^1|_{S^1}$ for $x\notin G$. 
\end{thm}

With this theorem, we are able to prove the following.
\begin{lem}\label{lem:reductionUsingOrponen}
Let $Y\subseteq\mathbb{R}^2$ be compact with $0<\mathcal{H}^d(Y)<\infty$ for some $d_y>1$, and let $d_x>1$ be given. Let $\mu = \mathcal{H}^s|_Y$. Let $A\subseteq\mathbb{N}$ be an oracle such that $A$ is a packing oracle for $Y$, $Y$ is computably compact relative to $A$ and $\mu$ is computable relative to $A$. Then, there is a set $G$ of Hausdorff dimension at most $d_x$ such that for all $x\in \mathbb{R}^2 - (\text{spt}(\mu)\cup G)$, and all $B\subseteq\N$, there exists some $y\in Y$ such that the pair $x, y$ satisfy conditions (C1)-(C4).
\end{lem}
\begin{proof}
Let $G$ be the set guaranteed by Orponen's radial projection theorem with respect to $Y$ and $\mu=\mathcal{H}^s|_Y$. Let $x\in \mathbb{R}^2 - (\text{spt}(\mu)\cup G)$ have effective Hausdorff dimension relative to $A$ greater than $d_x$, and define $N = \{e\in S^1: (\exists^\infty r) K_r^{x, A}(e)<r-4\log(r)\}$. As observed in the second author's previous paper, $\mathcal{H}^1|_{S^1}(N) = 0$. Orponen's theorem guarantees the absolute continuity of $\pi_{x\#}(\mu)$ with respect to $\mathcal{H}^1|_{S^1}$ for $x$ outside the exceptional set $G$, where $\pi_x(y) = \frac{y-x}{||y-x||}\in S^1$. Thus, $\mathcal{H}^s|_Y(\pi^{-1}_x(N)) = 0$. Now, let 
\begin{equation*}
M = \{y \in Y: \dim^A(y)\geq s \text{ and } K_r^{x, A,B}(y)>K_r^A(y) - 8 \log r \text{ for large enough }r\}
\end{equation*}
Again from \cite{Stull22c}, we have that $\mathcal{H}^s(M) = \mathcal{H}^s(Y)>0$ by assumption.\footnote{This result actually deals with $M = \{y \in Y: \dim^A(y)\geq s \text{ and } K_r^{x, A}(y)>K_r^A(y) - 8 \log r \text{ for large enough }r\}$, but since the statement only involves $x$ as an oracle, the same proof goes through with the oracle $x, B$.} Thus, $M - \pi_x^{-1}(N)$ has dimension $s$ and is in particular nonempty. Picking a $y$ in $M - \pi_x^{-1}(N)$, we may check the conditions: $x$ satisfies (C1) by assumption, we group together all $e$ not satisfying (C2) in the negligible set $N$, and $y$ satisfies its part of (C1) and (C3) by the definition of $M$. As for (C4), Lemma 31 in \cite{Stull22c} shows that (C1)-(C3) imply (C4), so the proof of this lemma is complete. 
\end{proof}

Now, we would like to drop the $\text{spt}(\mu)$ from the excluded set, which we will do by proving the following lemma:

\begin{lem}\label{lem:lastReductionLemma}
Let $Y\subseteq \mathbb{R}^2$ be a compact set such that $0<\mathcal{H}^s(Y)<\infty$ for some $s>1$, let $A$ be any oracle relative to which $Y$ is effectively compact, and let $d_x>1$ be given. Then there is a set $G\subseteq\mathbb{R}^2$ of Hausdorff dimension at most $d_x$ such that for every $x\in \mathbb{R}^2 - G$, and every $B\subseteq\N$, there is a $y\in Y$ such that the pair $x, y$ satisfies (C1)-(C4)
\end{lem}

\begin{proof}
Let $Y$ be as in the statement of the lemma, and write $Y = Y_1 \cup Y_2$ for disjoint compact $Y_1$, $Y_2$ such that $0<\mathcal{H}^s(Y_1)<\infty$ and $0<\mathcal{H}^s(Y_2)<\infty$. Let $\mu_1, \mu_2$ be $\mathcal{H}^s|_{Y_1}, \mathcal{H}^s|_{Y_2}$ respectively. Suppose $A_1$ and $A_2$ are effective compactness oracles for $Y_1$ and $Y_2$ respectively, $A_3$ is an effective compactness oracle for $Y$, and let $\hat{\mu}_i$ encode $\mu_i(Q)$ for each ball $Q$ with rational center and radius. Let $A$ be the join of $A_1, A_2, A_3, \hat{\mu}_1$, and $\hat{\mu}_2$. Then $A$ and $Y_i$ satisfy the hypotheses of the previous lemma for each $i$.

Let $G_1, G_2$ be the exceptional sets guaranteed by the lemma. Let $G = G_1 \cup G_2 \cup \{x\in\mathbb{R}^2: \dim^A(x)\leq d_x\}$. If $x\in\mathbb{R}^2\setminus G$, then by the previous lemma, there is some $y$ in either $Y_1$ or $Y_2$ satisfying conditions (C1)-(C4) (since $x$ can only be in the support of at most \emph{one} of $\mu_1, \mu_2$).  
\end{proof}

Now we are in a position to prove our main theorem.

\begin{T8}
Let $Y\subseteq \R^2$ be analytic such that $1<d_y =\dim_H(Y)$ and $D_y=\dim_p(Y)$. Let $X\subseteq \R^2$ be such that $1<d_x <\dim_H(X) $ and $D_x=\Dim_p(X)$. Then there is some $F \subseteq X$ of full dimension such that 
\begin{equation*}
    \dim_H(\Delta_x E) \geq d\left(1 - \frac{\left(D-1\right)\left(D-d\right)}{2D^2+\left(2-4d\right)D+d^2+d-2}\right),
\end{equation*}
for all $x\in F$, where $d=\min\{d_x, d_y\}$ and $D=\max\{D_x, D_y\}$. In particular, $\dim_H(X\setminus F)\leq d_x<\dim_H(X)$  Furthermore, if 
\begin{equation*}
D<\frac{(3+\sqrt{5})d-1-\sqrt{5}}{2}
\end{equation*}
Then $\dim_H(\Delta_x E)=1$.
\end{T8}

\begin{proof}
Let $Y$ and $X$ be as above, and let $B$ be the trivial oracle.\footnote{We need the oracle $B$ to make the \emph{packing} reduction go through; in the Hausdorff case, it can be removed.} $Y$ is analytic, thus for any $s<d_y=\dim_H(Y)$, there is a compact $Y_s$ such that $0<\mathcal{H}^s(Y_s)<\infty$. Let $\{s_i\}_{i\in\mathbb{N}} = d_y - \frac{1}{i}$. Let $Y_{s_i}$ be a sequence of such compact sets. The $Y_{s_i}$ are compact, so let $A_1, A_2, ...$ be oracles relative to which they are effectively compact. 

Now, we use the general point-to-set principle. Let $A_Y$ and $A_X$ be packing oracles for $Y$ and $X$ respectively, that is, oracles such that $\dim_P(Y) = \sup_{y\in Y} \Dim^{A_Y}(y)$ and likewise for $X$. Now, take the join of $A_Y, A_X, A_1, A_2, ...$, and call this new oracle $A$.

Note that this oracle retains all the desired properties of its constituents. In particular, every $Y_{s_i}$ is effectively compact relative to $A$, for all $y\in Y$ we have that $\Dim^{A}(y)\leq \dim_P(Y) =D_y$, and for all $x\in X$ we have that $\Dim^{A}(x)\leq \dim_P(X) =D_x$. Using $A$, we may apply the previous lemma to each $Y_{s_i}$, giving a corresponding exceptional set $G_i$. Observe that $\dim_H(X)>d_x\geq \dim_H(G_i)$, so the set of $x\in{X}$ for which there is some $y\in Y_{s_i}$ satisfying conditions (C1)-(C4) is nonempty for each sufficiently large $i$ (large enough that $s_i>1$). In fact, by the countable stability of Hausdorff dimension, $\dim_H(X)>d_x\geq \dim_H(\bigcup_{i\in\mathbb{N}} G_i)$. If we denote this union by $G$, then for all $x\in X-B$, there is a $y\in Y$ such that (C1)-(C4) hold with $A$ as the oracle.

We choose $A$ such that the desired packing dimension conditions hold relative to $A$ for \emph{any} $x\in X$, $y\in Y_{s_i}\subseteq(Y)$, so along with conditions (C1)-(C4), all the hypotheses of our effective theorem are satisfied. Applying this theorem when $D\geq\frac{(3+\sqrt{5})d-1-\sqrt{5}}{2}$, then, we obtain that 

\begin{equation*}
\dim^{A, x}(\vert x-y\vert) \geq (d-\frac{1}{i})\left(1 - \frac{\left(D-1\right)\left(D-(d-\frac{1}{i})\right)}{2D^2+\left(2-4(d-\frac{1}{i})\right)D+(d-\frac{1}{i})^2+(d-\frac{1}{i})-2}\right).
\end{equation*}
By Observation \ref{obs:effcompact}, since $Y_{s_i}$ is effectively compact relative to $A$, $\Delta_x Y_{s_i}$ is effectively compact relative to $(A, x)$. Finishing the proof, by Theorem \ref{thm:strongPointToSetDim}, we have that 

\begin{align*}
\dim_H(\Delta_x E) &\geq \dim_H(\Delta_x E_{s_i})\\
&= \sup_{y\in E_{s_i}} \dim^{A, x}(\vert x-y\vert)\\
&\geq (d-\frac{1}{i})\left(1 - \frac{\left(D-1\right)\left(D-(d-\frac{1}{i})\right)}{2D^2+\left(2-4(d-\frac{1}{i})\right)D+(d-\frac{1}{i})^2+(d-\frac{1}{i})-2}\right).
\end{align*}

Letting $i$ go to infinity completes the proof in this case. If we are in the case that $D<\frac{(3+\sqrt{5})d-1-\sqrt{5}}{2}$, then for some $i$ large enough, we have
\begin{equation*}
D<\frac{(3+\sqrt{5})(d-\frac{1}{i})-1-\sqrt{5}}{2}.
\end{equation*}
Then, repeating the above argument with $i$ at least this large gives the bound of 1 immediately and completes the proof.

\end{proof}

Assuming $E=X = Y$, we immediately have as a corollary that 

\begin{T1}
Let $E\subseteq \R^2$ be analytic such that $1<d <\dim_H(E)$. Then there is a subset $F \subseteq E$ of full dimension such that 
\begin{equation*}
    \dim_H(\Delta_x E) \geq d\left(1 - \frac{\left(D-1\right)\left(D-d\right)}{2D^2+\left(2-4d\right)D+d^2+d-2}\right),
\end{equation*}
for all $x\in F$, where $D = \dim_P(E)$. In particular, $\dim_H(E\setminus F)\leq d<\dim_H(E)$. Furthermore, if 
\begin{equation*}
D<\frac{(3+\sqrt{5})d-1-\sqrt{5}}{2}
\end{equation*}
Then $\dim_H(\Delta_x E)=1$.
\end{T1}

In a similar manner, we prove the following theorem. 

\begin{T6}
Let $Y\subseteq\R^2$ be analytic with $\dim_H(Y) > 1$ and $\dim_P(Y) < 2 \dim_H(Y)-1$. Let $X \subseteq \R^2$ be any set such that $\dim_H(X) > 1$. Then for all $x\in X$ outside a set of (Hausdorff) dimension one,
\begin{center}
    $\dim_H(\Delta_x Y) = 1$.
\end{center}
\end{T6}

    We can follow essentially the same argument as above, except now we only need the dimension of $x$ to be greater than 1, at the cost of also requiring that $\dim^A(y)$ and $\Dim^A(y)$ are close enough, as in the hypothesis of Proposition \ref{prop:yellowfulldimension}.
\begin{proof}
Let $Y$ and $X$ be as above, and again let $B$ be the trivial oracle. $Y$ is analytic, thus for any $s<\dim_H(Y)$, there is a compact $Y_s$ such that $0<\mathcal{H}^s(Y_s)<\infty$. Let $s$ be such that $\dim_p(Y)<2s-1<2\dim_H(Y)-1$, and let $A_s$ be an oracle relative to which $Y_s$ is effectively compact. Let $A_Y$ be a packing oracle for $Y$. Now, take the join of $A_Y, A_s$, and call this new oracle $A$.

Now, apply Lemma \ref{lem:lastReductionLemma} to $Y_{s}$ relative to $A$ with $d_x=1$. As above, conditions (C1)-(C4) are satisfied with an exceptional set of dimension at most 1. $Y_s\subseteq Y$, so for any $y\in Y_s$, $\Dim^A(y)\leq \dim_P(Y)$. Hence, $\Dim^A(y)<2\dim^A(y)-1$, and all the conditions for Proposition \ref{prop:yellowfulldimension} are satisfied, and we have as above that for $x$ not in our exceptional set,
\begin{align*}
\dim_H(\Delta_x Y) &\geq \dim_H(\Delta_x Y_s)\\
&= \sup_{y\in Y_s} \dim^{A, x}(||x-y||)\\
&=1
\end{align*}
\end{proof}

\bigskip

\subsection{Packing dimension of pinned distance sets}

\begin{T5}
Let $E\subseteq \R^2$ be analytic such that $\dim_H(E) > 1$. Then, for some $x\in E$,
\begin{equation*}
    \dim_P(\Delta_x E) \geq \frac{12 -\sqrt{2}}{8\sqrt(2)}\approx 0.9356.
\end{equation*}
\end{T5}
\begin{proof}
Let $E\subseteq \R^2$ be analytic such that $d = \dim_H(E) > 1$. Let $D = \dim_P(E)$. Since $E$ is analytic, there is a compact subset $F\subseteq E$ such that $0 < \mathcal{H}^s(F) < \infty$, for $s = \frac{d+1}{2}$. Let $\mu = \mathcal{H}^s_{\vert F}$. Let $A_1$ be an oracle relative to which $F$ is effectively compact. 

Let $A_2$ be a packing oracle for $F$, that is,  $\dim_P(F) = \sup_{x\in F} \Dim^{A_2}(x)$. Let $A$ be the join of $A_1$ and $A_2$.

Using Orponen's radial projection theorem, we see that there is an exceptional set $G_1$ such that the pushforward of $\mu$ under $\pi_x$ is absolutely continuous with respect to $\mathcal{H}^1_{\vert S^1}$ for all $x \notin G_1$. Let $G_2 = \{x\in F \mid \dim^A(x) < \frac{s+1}{2}\}$ and $G = G_1 \cup G_2$. Since $\dim_H(G_1) \leq 1$ and $\dim_H(G_2) \leq \frac{s+1}{2}< s$ we see that $\dim_H(G) < s$. We now choose $x \in F - G$ such that the set
\begin{center}
    $F^\prime = \{y \in F - G\mid \Dim^A(x)\leq \Dim^A(y)\}$
\end{center}
has positive $\mu$ measure, which is possible since we can choose some $x$ that has minimal or nearly minimal effective packing dimension. Let $B$ be a packing oracle for $\Delta_x E$. The proof of Lemma \ref{lem:reductionUsingOrponen} shows that there is a $y\in F^\prime$ such that the pair $x,y$ satisfies (C1)-(C4). Moreover, $\Dim^A(x)\leq \Dim^A(y)$.

We may therefore apply Theorem \ref{thm:effectivePackingThm}, which shows that
\begin{equation*}
\Dim^{A, x, B}(\vert x-y\vert)\geq \frac{3D^2 - D + 6}{8D},
\end{equation*}
where $D = \Dim^A(y)$. Since this is minimized when $D = \sqrt{2}$, we conclude that 
\begin{equation*}
\Dim^{A, x, B}(\vert x-y\vert)\geq \frac{12 -\sqrt{2}}{8\sqrt(2)}.
\end{equation*}

Our choice of $B$ and the point-to-set principle concludes the proof, since
\begin{align*}
\dim_P(\Delta_x E) &\geq \sup_{y\in E} \Dim^{B}(\vert x - y\vert)\\
&\geq \sup_{y\in E} \Dim^{x, A, B}(\vert x-y\vert)\\
&\geq _\frac{12 -\sqrt{2}}{8\sqrt(2)}.
\end{align*}
\end{proof}

\bigskip

\section{Regular pin sets give full dimension pinned distance sets}

Now, we consider a case that is not covered by our previous theorems. Our work thus far implies that if $Y$ is sufficiently regular  and $\dim_H(X)>1$, then $\dim_H(\Delta_x Y)=1$ for all $x$ in a full dimension subset of $X$. 

More surprisingly, however, the above holds even if just the \emph{pin} set $X$ is regular. In this case, we are able to deduce an essentially optimal effective projection theorem for $x$ which allows us to utilize arbitrarily long green intervals when partitioning $K^A_r(y)$. If we have access to these green intervals, we will never be forced to use a strictly teal interval\footnote{A non-green teal interval, so one with a growth rate of strictly less than 1}, implying that our partition is all-yellow and that $\dim^{x, A}(\vert x-y \vert) =1$.

We would like to perform the reduction to the classical result by finding some point $x$ that is regular, but there is a problem. Suppose $X$ is regular of dimension greater than 1 and $\dim_H(Y)>1$. In general, we cannot assume $X$ contains regular points relative to a given oracle. For instance, consider the set
\begin{center}
    $X=\{x\in\mathbb{R}^2: \Dim(x)=1.2\text{ and }\dim(x)<1.2\}$.
\end{center}
The packing and Hausdorff dimensions of $X$ will be 1.2, but, relative to any oracle $A$, any point will have effective Hausdorff dimension strictly less than 1.2, though it may have packing dimension \emph{exactly} 1.2. However, we can overcome this obstacle, due to the fact that $\dim_H(Y) > 1$. This fact implies that there a bound on the length of any green interval.

We start by stating and proving an effective projection theorem that holds when $x$ is sufficiently regular.
\begin{prop}\label{prop:alternateOptimalProjectionTheorem}
Let  $A\subseteq\N$, $x \in \R^2$, $e \in \mathcal{S}^1$, $\ve\in \Q^+$, $C, C^\prime\in\N$, and $t, r \in \N$. Suppose that $r$ is sufficiently large, and that the following hold.
\begin{enumerate}
\item[\textup{(P1)}] $1 < d_x < \dim^A(x)$.
\item[\textup{(P2)}] $ t \geq  \frac{r}{C} $.
\item[\textup{(P3)}] $K^{x, A}_s(e) \geq s - C^\prime\log s$, for all $s \leq t$. 
\end{enumerate}
Then there exists some $\ve_x$ depending only on $d_x$ and $C$ such that $\Dim^A(x)-d_x<\ve_x$ implies that 
\begin{equation*}
K^A_r(x \,|\, p_e x, e) \leq K^A_r(x) - r + \ve r.
\end{equation*}
\end{prop}

The key of this result is that the $\ve_x$ we need in the almost-regularity condition does not depend on $\ve$. The way we will apply this theorem, $C$ is going to depend on $d_y$. So, given a lower bound for the effective Hausdorff dimension of points $x$ and $y$, $\ve_x>0$ expresses how close to regular $x$ needs to be in order to apply this theorem, and that required closeness is \emph{fixed} given sets $X, Y$ when we perform the reduction. 

The idea of the proof is quite simple. When we partition in $x$, we were allowed to use intervals up to length $t$. By picking $r$ large enough that the complexity function lies very close to the line $d_x s$,\footnote{How close it needs to be to this line for the argument to go through depends on $C$, how close we are allowed to assume it is depends on $\ve_x$ since the point is not exactly regular. Thus, we choose $\ve_x$ based in part on $C$.} we can show that it is impossible for there to be any red-green-blue sequences past some precision $\ve r$, since this would imply the existence of a green block of length at least $t$. Thus, $[1, r]$ is almost entirely covered by red and green intervals, and we obtain an essentially all-yellow partition of it. Formally, 

\begin{proof}
Let $x$ satisfying the conditions of the proposition for some $\ve_x$ be given. Let $C$, and $\ve$  be given and choose $\ve^\prime$ depending on $\ve_x$, $d_x$, and $C$ in a manner which we will detail shortly. Now let $r$ be large enough that $d_x s-\ve^\prime s\leq K^A_s(x)\leq d_x s+\ve^\prime s$ for all $s\geq \sqrt{r}$. Note that we must have $\ve^\prime>\ve_x$. Let $t\geq \frac{r}{C}$ be given, and let $\hat{\mathcal{P}}=\hat{\mathcal{P}}(x, r, t)$ be the partition of $[1, r]$ by red, blue and green intervals considered in section 3.1. We will show that if $\ve_x$ is sufficiently small depending only on $C$ and $d_x$, then the desired conclusion holds.

On one hand, any red-green-blue sequence has a green block of length at least $t$; this was one of the key properties of this partition in section 3. On the other hand, since the slope is $1$ on green intervals, if we have $K^A_{r-t}(x)=(d_x +\ve^\prime)(r-t)$, (its maximum possible value) and that $[r-t, r]$ is green, then 
\begin{align*}
K^A_r(x)&\leq(d_x +\ve^\prime)(r-t) + t\\
&\leq (d_x+\ve^\prime) r - (d_x-1-\ve^\prime)\frac{r}{C}
\end{align*}
So, if $\ve^\prime$ is small enough that 
\begin{equation}\label{eq:newboundOnOtherEpsilon}
-\frac{d_x}{C}+2\ve^\prime -\frac{\ve^\prime}{C}-\frac{1}{C}<0,
\end{equation}
we have a contradiction, since by assumption $K^A_r(x)\geq d_x r-\ve^\prime r$. If we attempted to place the left endpoint of a green block of a red-green-blue sequence at any $\sqrt{r}\leq s\leq r$, we would clearly have the same contradiction, in fact, the green interval would be forced to be even shorter. Hence, there are no red-green-blue sequences such that the green block is contained in $[\sqrt{r}, r]$. Provided that $\ve_x$ satisfies equation (\ref{eq:newboundOnOtherEpsilon}), we can choose such an $\ve^\prime>\ve_x$ also satisfying it.

Now, we want to determine when, given $\ve$, the interval $[\sqrt{r}, \frac{\ve}{2} r]$ is forced to intersect at least one red interval. For convenience, let $\ve^{\prime\prime} r = \sqrt{r}$. Then, we will choose $r$ large enough to satisfy this condition and large enough that $\ve^\prime$ satisfies (\ref{eq:newboundOnOtherEpsilon}). Notice that $[\ve^{\prime\prime} r, \frac{\ve}{2} r]$ intersecting some red interval will imply there are no blue intervals in $[\frac{\ve}{2}, r]$, or else we would have the green block of a red-green-blue sequence contained in $[\ve^{\prime\prime} r, r]$. $[\ve^{\prime\prime} r, \frac{\ve}{2} r]$ must intersect a red interval when the average growth rate of $K^A_s(x)$ on this interval is strictly greater than 1. We can easily bound that growth rate as follows
\begin{equation}
\dfrac{K^A_{\frac{\ve}{2}r}(y)-K^A_{\ve^{\prime\prime} r}(y)}{(\frac{\ve}{2} -\ve^{\prime\prime})r}\geq\dfrac{(d_x-\ve^{\prime\prime})\frac{\ve}{2}r-(d_x+\ve^{\prime\prime})\ve^{\prime\prime} r}{(\frac{\ve}{2} -\ve^{\prime\prime})r}.
\end{equation}
Since $d_x>1$, for any $\ve>0$, there exists some $\ve^{\prime\prime}$ small enough that the right hand side is greater than 1. Hence, for $r$ sufficiently large, $[\frac{\ve}{2}r, r]$ is covered entirely by red and green intervals. 

Just as in the $S=0$ case of section 3.2, we cite the result that an interval covered only by red and green intervals in $\hat{\mathcal{P}}$ can be covered by an all-yellow $3C$-admissible partition. Summing over this partition using Lemma \ref{lem:boundGoodPartitionProjection} relative to $A$ with respect to $\frac{\ve}{2}$ then yields

\begin{equation}
K^A_r(x \mid p_e x, e) \leq K^A_r (x) - r + \ve r,
\end{equation}
which completes the proof.
\end{proof}

Equipped with this proposition, we define a new partition of $[1, r]$ for $K^A_s(y)$. The key is that this partition will, after a certain small precision, only use yellow intervals. Some of these yellow intervals will be very long green intervals, and we will use the above proposition to prove an analogue of Lemma \ref{lem:goodtealgrowth} on them. Once we are able to sum over the intervals in our partition, we are essentially done. We now define this partition inductively. 

Set $r_0=r$ and assume we have defined the sequence up to $r_i$. Take a good partition of $[1, r_i]$. Let $a$ denote the minimal real such that $[a, r_i]$ is the union of yellow intervals whose lengths are at most doubling. If $a < r_i$ add all these yellow intervals to the partition, and let $r_{i+1} = a$. Otherwise, let $r_{i+1}$ be the smallest precision such that $[r_{i+1}, r_i]$ is green, i.e., such that $f(r_i) = f(r_{i+1}) + r_i - r_{i+1}$. Note that such an $r_{i+1}$ exists in this case, since $r_i$ is the right endpoint of a teal interval. Greedily combine intervals that have a less than doubling union, re-index the intervals, and denote this partition by $\mathcal{P}$. It is easy to see that $2 r_{i+2} <r_i$. As a result, once we have established that we can use these long green intervals in our sum, the conclusion of Lemma \ref{lem:boundGoodPartitionDistance} follows immediately.

Now, we prove a lower bound for $r_{i+1}$ on these green intervals (depending on $d_y$) on when $r_i$ is sufficiently large. 

\begin{lem}\label{lem:regularSLowerBound}
Let $0<\ve<\frac{d_y-1}{2}$ be given, and suppose $r_i$ is large enough that for all $s>\frac{d_y-1}{8} r_i$, we have $d_y s - \ve s\leq K_s^A(y)\leq 2 s + \ve s$. Then
\begin{equation}
r_{i+1}\geq \frac{d_y-1-\ve}{3+\ve}r_i\geq \frac{d_y-1}{8} r_i
\end{equation}
\end{lem}
\begin{proof}
The desired result is an immediate consequence of the fact that the line $K^A_{r_i}(y)-(r_i-s)$ cannot intersect $K^A_s(y)$ for any $s<\frac{d_y-1-\ve}{3+\ve}r_i$, hence, any green interval with right endpoint $r_i$ has to have left endpoint larger than $\frac{d_y-1-\ve}{3+\ve}r_i$. Now, assume $s$ is such that $K^A_{r_i}(y)-(r_i-s)=K^A_s(y)$. The aforementioned fact is the consequence of a simple calculation. 
\begin{equation}
K^A_{r_i}(y)-(r_i-s)\geq d_y r_i-\ve r_i - (r_i-s)
\end{equation}
whereas $K_s^A(y)\leq 2s +\ve s$. Setting the bounds equal yields $r_{i+1}\geq s\geq\frac{d_y-1-\ve}{3+\ve}r_i$ and thus completes the proof.
\end{proof}

Now, we are able to state and prove the analogue of Lemma \ref{lem:goodtealgrowth}. 
\begin{lem}\label{lem:greenBoundRegularPins}
Suppose that $[r_{i+1}, r_i]$ is a green interval in our constructed partition. For any $\ve>0$, provided that $r_{i+1}$ is sufficiently large, we have 
\begin{equation}
    K^{A,x}_{r_i, r_i, r_{i+1}}(y \mid \vert x - y \vert, y) \leq \ve r_i.
\end{equation}
Therefore,
$K^{A,x}_{r_i, r_{i+1}}(\vert x - y\vert \mid \vert x - y\vert ) \geq K^{A,x}_{r_i, r_{i+1}}(y\mid y) - \ve r_{i}$.
\end{lem}
As before, the strategy is to use our projection theorem to verify that the hypotheses of Lemma \ref{lem:pointDistance} hold, which immediately implies the desired bound.

\begin{proof}
    Let some small rational $\ve>0$ be given, and assume $r_{i+1}$ is sufficiently large. Let $\eta$ be the rational such that $\eta r= K^A_r(y)-4\ve$. Let $G=D(r, y, \eta)$ be the oracle of Lemma \ref{lem:oracles} relative to $A$. Let $z\in B_{2^{-r_{i+1}}}(y)$ and satisfying $\vert x-y\vert=\vert x-z\vert$ be given. 

Letting $s=\vert y-z\vert$, we again have two cases. The $s\geq \frac{r_i}{2}-\log r_i$ case is identical to the previous proof, so we assume $s<\frac{r_i}{2}-\log r_i$. Lemma 29 applied relative to $(A, G)$ implies that 

\begin{equation}
K^{A, G}_r(z) \geq K^{A, G}_s(y) + K^{A, G}_{r_i-s, r_i}(x\mid y) - K^{A, G}_{r_i-s}(x\mid p_{e^\prime} x, e^\prime) - O(\log r). 
\end{equation}
\noindent To bound the projection term, we need to apply Proposition \ref{prop:alternateOptimalProjectionTheorem} with respect to $x$, $e^\prime$, $\frac{\ve}{2}$, the constant $C$ such that $\frac{1}{C}= \frac{d_y-1}{16}$, $t=s$, and $r=r_i-s$. 

Again, we have that $r_{i+1}-1<s<\frac{r_i}{2} - \log r_i$, since $z$ is assumed to be within $2^{-r_{i+1}}$ of $s$. Clearly, then, we can take $r_i-s$ to be sufficiently large. By assumption, $1<d_x=D_x$, and by \ref{lem:regularSLowerBound}, we have that 
\begin{align*}
s&\geq\frac{d_y-1}{8}r_i-1\\
&\geq \frac{d_y-1}{8}(r_i-s)-1\\
&\geq \frac{d_y-1}{16}(r_i-s)\\
&= \frac{r_i-s}{C}
\end{align*}
 Thus, we can examine sufficiently large precisions and conditions (P1) and (P2) are satisfied. As before, (P3) is satisfied using (C2). After using some of the properties of the oracle $G$ and taking $r_i-s$ to be sufficiently large, we can apply the projection theorem relative to $A$, as below:

\begin{align*}
K^{A, G}_{r_i}(z) &\geq K^{A, G}_s(y) + K^{A, G}_{r_i-s, r_i}(x\mid y) - K^{A, G}_{r_i-s}(x\mid p_{e^\prime} x, e^\prime) - O(\log r)\\
&\geq K^{A, G}_s(y) + K^{A}_{r_i-s, r_i}(x\mid y) - K^{A, G}_{r_i-s}(x\mid p_{e^\prime} x, e^\prime) - O(\log r)\\
&\geq K^{A, G}_s(y) + K^{A}_{r_i-s}(x) - K^{A, G}_{r_i-s}(x\mid p_{e^\prime} x, e^\prime) - O(\log r)\\
&\geq K^{A, G}_s(y) + d_x (r_i-s) -\frac{\ve}{2} r_i - K^A_{r_i-s}(x\mid p_{e^\prime} x, e^\prime)\\
&\geq K^{A, G}_s(y) + d_x (r_i-s) - d_x(r_i-s) + (r_i-s) - \ve r_i\\
&= K^{A, G}_s(y) + (r_i-s) - \ve r_i
\end{align*}
Because of the choice of $\eta$, for small enough $\ve$ we can guarantee that $K^{A, G}_s(y)=K^A_s(y) - O(\log r_i)$, since $s<\frac{r_i}{2}$. Hence, 
\begin{equation*}
    K^{A, G}_{r_i}(z)\geq K^{A}_s(y) + (r_i-s) - 2 \ve r_i.
\end{equation*}
Finally, using the fact that these intervals are green, hence teal, we have 
\begin{align*}
    K^{A, G}_{r_i}(z)&\geq K^{A}_s(y) + (r_i-s) - 2 \ve r_i\\
    &\geq K^{A}_{r_i}(y)-(r_i-s) + (r_i-s) - 2 \ve r_i\\
    &\geq (\eta +\ve) r_i\\
\end{align*}
Thus, the conditions for Lemma \ref{lem:pointDistance} hold and we apply it, completing the proof. 
\end{proof}

Now, we are able to prove the main effective theorem of this section. As indicated, since we can now sum over our green intervals, we can begin with the conclusion of Lemma \ref{lem:boundGoodPartitionDistance}.

\begin{thm}\label{thm:secondEffectiveTheoremPinnedSet}
There is some $\ve_x$ sufficiently small depending only on $d_x$ and $d_y$ that, supposing $x, y\in\R^2$, $e = \frac{y-x}{\vert y-x\vert}$, and $A, B\subseteq\N$  satisfy the following.
\begin{itemize}
\item[\textup{(C1)}] $1<d_y<\dim^A(y), d_x<\dim^A(x)\leq \Dim^A(x) <D_x$.\footnote{Our condition (C1) is slightly different than before; this is just to make the reduction easier. Specifically, it will be helpful if we now think of $d_y$ as a strict lower bound on the dimension of points $y$. Note that this change will not complicate the application of Proposition \ref{lem:lastReductionLemma}, since we will use it on a compact set with dimension greater than $d_y$.}
\item[\textup{(C2)}] $K^{x,A}_r(e) = r - O(\log r)$ for all $r$.
\item[\textup{(C3)}] $K^{x,A, B}_r(y) \geq K^{A}_r(y) - O(\log r)$ for all sufficiently large $r$. 
\item[\textup{(C4)}] $K^{A}_r(e\mid y) = r - o(r)$ for all $r$.
\end{itemize}
Then, $\Dim^A(x)-\dim^A(x)<\ve_x$ implies
\begin{center}
$\dim^{x,A}(\vert x-y\vert) =1$
\end{center}
\end{thm}

\begin{proof}
Let $\ve>0$ be given. Let $r$ be large enough that applying Lemma \ref{lem:boundGoodPartitionDistance} in the form of equation (\ref{eq:distancepartitionbound}), we have

\begin{equation}
    K^{A, x}_r(\vert x - y\vert) \geq K^A_r(y) - \sum\limits_{i\in \textbf{Bad}} K^A_{a_{i+1}, a_{i}}(y \mid y) - (a_{i+1} - a_i) - \frac{\ve}{2} r.
\end{equation}
Recalling that the complexity grows at an average rate of exactly 1 on green intervals, thus implies
\begin{align}
    K^{A, x}_r(\vert x - y\vert) &\geq K^A_r(y) - \sum\limits_{i\in \mathcal{P}} K^A_{a_{i+1}, a_{i}}(y \mid y) - (a_{i+1} - a_i) - \frac{\ve}{2} r\\
    &\geq K^A_r(y)-K^A_r(y) + r - \ve r\\
    &= (1-\ve) r
\end{align}
Taking $\ve$ as small as desired then letting $r$ go to infinity completes the proof. 
\end{proof}

Finally, we state and prove the classical result.

\begin{thm}
Let $X\subseteq\mathbb{R}^2$ be such that $1<d_x<\dim_H(X)=\dim_P(X)$ and let $Y\subseteq \mathbb{R}^2$ be analytic and satisfy $1<d_y<\dim_H(Y)$. Then there is some $F\subseteq X$ such that, 
\begin{center}
    $\dim_H(\Delta_x Y)=1$
\end{center}
\noindent for all $x\in F$. Moreover, $\dim_H(X\setminus F)<\dim_H(X)$. 
\end{thm}

\begin{proof}
  Let $X$ and $Y$ be as above, and let $B$ be the trivial oracle. $Y$ is analytic, thus for any $d_y<s<\dim_H(Y)$, there is a compact $Y_s$ such that $0<\mathcal{H}^s(Y_s)<\infty$. Fix some such $s$ and a corresponding $Y_s$, and let $A_s$ be its effective compactness oracle.

Now, we use the general point-to-set principle. Let $A_Y$ and $A_X$ be packing oracles for $Y$ and $X$ respectively. Now, take the join of $A_X, A_Y$, and $A_s$ and call this new oracle $A$.

 Relative to $A$, we may apply Lemma \ref{lem:lastReductionLemma} to $Y_{s}$, giving a corresponding exceptional set $G$. Note that we choose $A$ such that the desired packing dimension conditions hold relative to $A$ for \emph{any} $x\in X$. In light of this fact and observing that $\dim_H(X)>d_x\geq \dim_H(G)$, the set of $x\in{X}$ for which there is some $y\in Y_{s}$ satisfying conditions relative to $A$ (C1)-(C4) is nonempty. Then for all $x\in X-G$, there is a $y\in Y_{s}$ such that (C1)-(C4) hold with $A$ as the oracle.

 Now, further refine $X-G$ by removing all the points such that $\Dim_P(X) - \dim^A(x)\geq \ve_x$ for the $\ve_x$ in Theorem \ref{thm:secondEffectiveTheoremPinnedSet}. Since $A$ is a packing oracle for $X$, this implies that the remaining points satisfy $\Dim^A(x)- \dim^A(x)<\ve_x$ Letting $F$ denote $X-\left(G\cup \{x: \Dim_P(X) - \dim^A(x)\}\right)$ immediately implies that all $x\in F$ satisfy the requirements of Theorem \ref{thm:secondEffectiveTheoremPinnedSet}, and furthermore that $\dim_H(X\setminus F)<\dim_H(X)$. 

Finally, applying the effective theorem yields that $\dim^{A, x}(\vert x-y\vert)=1$ for such a pair $x, y$. By Observation \ref{obs:effcompact}, since $Y_{s}$ is effectively compact relative to $A$, $\Delta_x Y_{s}$ is effectively compact relative to $(A, x)$. Finishing the proof, by theorem \ref{thm:strongPointToSetDim}, we have that 

\begin{align*}
\dim_H(\Delta_x Y) &\geq \dim_H(\Delta_x Y_{s})\\
&= \sup_{y\in Y_{s}} \dim^{A, x}(\vert x-y\vert)\\
&= 1
\end{align*}
\end{proof}

\bibliographystyle{amsplain}
\bibliography{ibpds}

\end{document}